\documentclass[a4paper,11pt]{article}
\usepackage[utf8]{inputenc}
\usepackage{amsmath, amssymb, amsthm}
\usepackage{amsopn}
\usepackage{lmodern} 
\usepackage{a4wide} 
\usepackage{color}
\usepackage{graphicx}
\usepackage{hyperref}
\usepackage{enumitem}
\usepackage{mathtools}
\usepackage{esint}
\usepackage{verbatim} 
\usepackage[scr=kp]{mathalpha} 
\newcommand{\bR}{\mathbb R} 
\newcommand{\bS}{\mathbb S} 

\newcommand{\cR}{\mathcal R} 
\newcommand{\cF}{\mathcal F} 
\newcommand{\cE}{\mathcal E}
\newcommand{\cL}{\mathcal L}
\newcommand{\cH}{\mathcal H}  
\newcommand{\cD}{\mathcal D}
\newcommand{\loc}{{\textit{loc}}}
\newcommand{\dd}{\mathrm d}  
\newcommand{\weaklystar}{\stackrel{*}{\rightharpoonup}}

\DeclareMathOperator{\Tr}{Tr}

\DeclareMathOperator{\Div}{div}

\newcommand{\D}{D} 

\newcommand{\mres}{
	\,\raisebox{-.127ex}{\reflectbox{\rotatebox[origin=br]{-90}{$\lnot$}}}\,
}

\newcommand{\K}{K}  

\setlist[enumerate]{leftmargin=.5in}
\setlist[itemize]{leftmargin=.5in}

\newtheorem{prop}{Proposition}[section]
\newtheorem{thm}[prop]{Theorem}
\newtheorem{lemma}[prop]{Lemma}

\theoremstyle{remark}
\newtheorem{remark}{Remark}

\title{Inclusion and estimates for the jumps of minimizers in~variational denoising\thanks{M{\L}~was supported by grant {\tt 2020/36/C/ST1/00492} of the National Science Centre, Poland. AC~acknowledges the support of the ``France 2030'' funding ANR-23-PEIA-0004 (``PDE-AI'').}}

\author{Antonin Chambolle\thanks{CEREMADE, CNRS, Université Paris-Dauphine, PSL University, Paris, MOKAPLAN,  INRIA, Paris, France ({\tt chambolle@ceremade.dauphine.fr}).}
	\and Micha{\l} {\L}asica\thanks{Institute of Mathematics of the Polish Academy of Sciences, Warsaw, Poland ({\tt mlasica@impan.pl}).}}

\begin{document}

\maketitle

\begin{abstract}
  We study stability and inclusion of the jump set of minimizers of convex denoising functionals, such as the celebrated ``Rudin--Osher--Fatemi'' functional, for scalar or vectorial
  signals. We show that under mild regularity assumptions on the data fidelity term and the regularizer, the jump set of the minimizer is essentially a subset of the original jump set. Moreover, we give an estimate on the magnitude of jumps in terms of the data. This extends old results, in particular of the first author (with V.~Caselles and M.~Novaga) and of T.~Valkonen, to much more general cases. We also consider the case where the original datum has unbounded variation, and define a notion of its jump set which, again, must contain the jump set of the solution.
\end{abstract}



\section{Introduction}

The Total Variation regularizer was proposed for image denoising in~\cite{ROF} and has
become popular for its simplicity and its ability to recover edges and discontinuities in
the restored images. Even if it is largely outdated and has much lower performances
than non-local~\cite{BM3D,Buadesetal,Lebrunetal},
(learned) patches and dictionary-based~\cite{EladAharon,ZoranWeiss} or neural network based~\cite{JainSeung} techniques, it remains useful as a regularizer for large
scale inverse problems (sometimes combined with machine learning
and plug-n-play type~\cite{PnP} methods, see for instance~\cite{Wang:23}), as
it is convex and relatively simple to optimize, in particular in combination 
with other (ideally also convex) terms.

An interesting question, answered first in~\cite{Casellesetal2007}, is whether
the total variation-denoising method
can create spurious structures and discontinuities, or if the edge set of the
original image is preserved. Precisely, given $f\in BV(\Omega)$ a (scalar) function with
bounded variation, representing the grey-level values of an image
defined in a domain $\Omega\subset\bR^m$ ($m$ an integer,  $2$ or $3$ in most applications), and
with jump set $J_f$ (see the precise definition in Sec.~\ref{sec:SJ}), one considers $u$ which solves:
\begin{equation}
  \min_u \int_\Omega |Du| + \frac{1}{2}\int_\Omega |u-f|^2. 
  \label{eq:ROF}
\end{equation}
The main result of~\cite{Casellesetal2007} asserts that $J_u\subset J_f$ (up to a
set negligible for the surface measure); in addition,
$u^+-u^-\le f^+-f^-$ a.\,e.~on the jump set of $u$. 
It is also deduced that  the $L^2$-gradient flow of
the total variation, starting from an initial function $u(0)\in L^{m/(m-1)}(\Omega)$,
has a diminishing jump set:
$s>t>0\Rightarrow J_{u(s)}\subset J_{u(t)}$ (more precise results are found in~\cite{CasellesJalalzaiNovaga}). This is generalized to some integrands (such as
the graph area, anisotropic total variations) already in~\cite{Casellesetal2007},
and further variants (including strictly convex data terms) in~\cite{jalalzai2014discontinuities,JalalzaiChambolle}; see also~\cite{jalalzai}.
The approach in the above mentioned papers consists in comparing the curvatures
of the level sets of minimizers. One shows that at (approximate) continuity points
of $f$, these curvatures are determined by the level, and ordered in a way which excludes the
possibility that the boundaries of two different level sets coincide.
The technique is relatively simple and elegant, and even allows to derive
basic regularity results away from the jumps~\cite{Mercier,Casellesetal_regularity},
but it is restricted to the scalar case.

An alternative approach was proposed by T.~Valkonen in~\cite{Valkonen2015}. It does not involve level sets of the minimizer, and therefore is not limited to the scalar case. In particular, the case where $u,f$ are vector valued and the total variation is defined by means of the Frobenius norm of the matrix $Du$ should enter the framework of~\cite{Valkonen2015}, even if this does not seem to be explicit in the literature. One reason for this is the relative complexity of the criterion in~\cite{Valkonen2015} (the \textit{double-Lipschitz comparability condition}), which is not always straightforward to check in practice, and the technicality of the papers~\cite{Valkonen2015,Valkonen2017}, 
which might have made them less accessible
to non-specialists, despite their interest and originality.

In this new study, we introduce a general approach for addressing the issue of jump
inclusion and control in variational denoising problems modeled on \eqref{eq:ROF}.
Essentially, we show that jump inclusion occurs when the regularizer is differentiable
with respect to an elementary class of inner variations of the solution, and derive an estimate on the magnitude of the jump (see for instance~\cite{CasellesJalalzaiNovaga}). This conclusively demonstrates that the jump inclusion property does not require any particular structure of the regularizer, but only its (mild) regularity. 
The differentiability assumption can roughly be viewed as a relaxation of double-Lipschitz comparability of \cite{Valkonen2015}, at least in the context of convex regularizers. It is satisfied by many regularizers appearing in imaging literature,
such as the Frobenius or (more surprisingly) the Nuclear (or Trace) Norm-based total variation in a vectorial
setting\footnote{Precisely,
it holds for Schatten-type norms, but not Ky-Fan type norms such as the Spectral Norm.}
(see for instance~\cite{NaturalTV,CollaborativeTV}).  Interestingly, while the extension of Valkonen's approach
to higher order regularizers, addressed in~\cite{Valkonen2017}, excludes
the ``Total
Generalized Variation'' (TGV) of~\cite{TGV2010}, a relatively simple modification
of our proof allows to show jump inclusion in a slightly regularized version
of that case, at least
whenever the solution $u$ is bounded (which can be enforced by a box type constraint
in the minimization). The result for the exact ``TGV'' case remains open and, if
true, probably requires a mix of our techniques and the ideas in~\cite{Valkonen2017},
which address successfully other types of inf-convolution based regularizers.

Our approach is based on a very simple observation: at a jump point, the data term (such as
the squared norm in~\eqref{eq:ROF}) will have different left and right derivatives along
inner variations orthogonal to the jump, so that, if the regularizer is differentiable, some
inequality is derived which involves only the data term. The idea can be illustrated by an elementary 1D example: consider $\Omega=]\!-\!1,1[\subset\bR$, $f\in BV(\Omega)$, and let $u$
minimize~\eqref{eq:ROF}. Consider then $\bar x\in J_u$, with $u^+(\bar x)>u^-(\bar x)$. Without
loss of generality we assume $u^+(\bar x)$ is the right-sided limit of $u$ at $\bar x$.
Denote $f^+(\bar x)$ the right-sided limit of $f$, and $f^-(\bar x)$ the left-sided limit (with possibly $f^+(\bar x)\le f^-(\bar x)$).
Then, if $\varphi$ is a smooth approximation of $\chi_{[\bar x-\delta,\bar x+\delta]}$, for $\delta>0$ small, for $\tau\in ]0,\delta[$ one has:
\begin{align*}
&    \int_\Omega (u(x+\tau \varphi(x))-f(x))^2 - (u(x)-f(x))^2 \dd x \approx \tau [(u^+(\bar x) - f^-(\bar x))^2 - (u^-(\bar x) - f^-(\bar x))^2], \\[2mm]
&    \int_\Omega (u(x-\tau \varphi(x))-f(x))^2 - (u(x)-f(x))^2 \dd x \approx \tau [(u^-(\bar x) - f^+(\bar x))^2 - (u^+(\bar x) - f^+(\bar x))^2]
\end{align*}
as $\tau\to 0$.
On the other hand, since the total variation of $u(x\pm \tau\varphi(x))$ is the same as the total variation of $u$, thanks to minimality of $u$ in~\eqref{eq:ROF} we deduce, sending $\tau\to 0$:
\begin{align*}
&    (u^+(\bar x) - f^-(\bar x))^2 - (u^-(\bar x) - f^-(\bar x))^2 \ge 0 ,\\
&    (u^-(\bar x) - f^+(\bar x))^2 - (u^+(\bar x) - f^+(\bar x))^2 \ge 0,
\end{align*}
that is: $(u^+(\bar x)-u^-(\bar x))(u^+(\bar x)+u^-(\bar x)- 2f^-(\bar x))\ge 0$
and $(u^+(\bar x)-u^-(\bar x))(u^+(\bar x)+u^-(\bar x)- 2f^+(\bar x))\le 0$.
We deduce that 
\[f^-(\bar x)\le \frac{u^+(\bar x)+u^-(\bar x)}{2}\le f^+(\bar x)\]
(and in particular $f^-(\bar x)\le f^+(\bar x)$), so
that either $\bar x \in J_f$, or $(u^++u^-)/2=f$ at $\bar x$. This is elementary, and almost the conclusion we would like to reach.

Actually proving the jump set inclusion (and an estimate on the jump) in any dimension, following the same idea, is not much harder but requires a more subtle choice of the variation. The solution is found  in Valkonen's work~\cite[Sec.~6]{Valkonen2015}, which uses a competitor for the minimization problem given by a convex  combination of the minimizer itself and its inner variation, see Lemma~\ref{lem:inner_diff} below. We show  (by a much simpler argument/calculation than in~\cite{Valkonen2015}) that together with the differentiability of the regularizer along inner variations, it allows to get a general estimate on the jump of $u$. This is done in Section~\ref{sec:main} (Theorem~\ref{thethm}). We also compare there our differentiability assumption with the double-Lipschitz comparability of \cite{Valkonen2015}.

Further (Sec.~\ref{sec:DiffReg}), we discuss general regularizers which satisfy the assumptions for our main result to hold. In particular, we find that the Frobenius or Nuclear-norm based Total Variations for vector-valued images meet our differentiability hypothesis (Section~\ref{sec:examples}). 
In Section~\ref{sec:Bounded} we discuss conditions which ensure that 
the solution $u$ to our variational problems, in the unconstrained case, are locally bounded.
In Section~\ref{sec:TGV} (Theorem \ref{thm:inf_conv}) we show how a small adjustment of the proof extends the result to the inf-convolution type regularizers such as smoothed variants of the Total Generalized Variation (TGV)~\cite{TGV2010}, studied in~\cite{Valkonen2017} again with a more complicated approach, and only partial conclusions. 
In that latter setting, our approach seems too simple and might need to be enriched with some of the ideas of~\cite{Valkonen2017} to be able to reach a full conclusion, but this remains difficult.

In Theorems \ref{thethm} and \ref{thm:inf_conv}, in line with previous results on the subject mentioned earlier, we work under the assumption that the noisy datum $f$ is a $BV$ function. However, in the case of highly oscillating noise, $f$ can have unbounded variation. We treat this general situation in Section \ref{sec:nonBV} by introducing a relaxed notion of jump set $\widetilde{J}_f$ for $f$ (as the set of points where $f$ differs significantly on both sides of a hyperplane) for which we can still show that it must contain the jump set of the solution. The set $\widetilde{J}_f$ has Lebesgue measure $0$ and essentially coincides with the usual jump set if $f$ is in $BV$. In Section~\ref{sec:experiment} we provide, as an illustration, an example where $f$ if is the sum of a $BV$ function and a bounded oscillating noise, and show the reconstruction with various types of color total variations, as introduced in~\cite{NaturalTV}.

\section{Preliminaries}
\subsection{General notation}
We will consider $\bR^n$-valued functions, $n\ge 1$, defined on some
open subset $\Omega$ of $\bR^m$, $m\ge 1$ (most of the proofs are written for $m\ge 2$,
yet the case $m=1$ follows by trivial simplification).
Given $x\in \bR^m$, $r>0$, $\nu\in\bS^{m-1}$ a unit vector, one denotes:
\begin{equation}\label{eq:notation}
\begin{aligned}
  &  B_r(x) = \{y\in\bR^m: |y-x|<r\},  \quad B_r^\pm(x,\nu) = \{y \in B_r(x)\colon \pm\nu \cdot (y-x) \geq 0\}, \\
& B^{m-1}_r(x,\nu) = B_r(x)\cap (x+\nu^\perp),  \quad Q_r(x, \nu) = B^{m-1}_r(x, \nu) + ]\!-r,r[\nu, \\
&Q^+_r(x, \nu) = B^{m-1}_r(x, \nu) + [0,r[\nu,  
\quad Q^-_r(x, \nu) = B^{m-1}_r(x, \nu) + ]\!-r,0]\nu.
\end{aligned}
\end{equation}
where $|\cdot|$ is the standard Euclidean norm.

For $\mu$ a Radon measure and $k\le m$, we define for any $x\in\Omega$ the $k$-dimensional density of $\mu$ at $x$ as the limit:
\[
\Theta^{k}(\mu,x)=\lim_{r\to 0} \frac{\mu(B_r(x))}{\omega_k r^k},
\]
when it exists. Here, $\omega_k$ is the volume of the unit ball of dimension $k$.

\subsection{The approximate discontinuity  set and the jump set}\label{sec:SJ}
Let $w \in L^1_{loc}(\Omega)^n$. Following \cite[Definition 3.63]{afp}, we say that $w$ has an approximate limit at $x \in \Omega$ if there exists $z \in \bR^n$ such that 
\begin{equation} \label{apprlim_def}
\lim_{r \to 0^+} \fint_{B_r(x)} |w(y) - z| \dd y =0. 
\end{equation}
If no such $z$ exists, $x$ is called an approximate discontinuity point of $w$. The set of all approximate discontinuity points of $w$ is denoted $S_w$. It is well known that $\cL^m(S_w)=0$ \cite[Proposition 3.64]{afp}. 

On the other hand, if there exist $\nu_w \in \bS^{m-1}$, $w^\pm \in \bR^n$, $w^- \neq w^+$, such that 
\begin{equation} \label{jump_def}
\lim_{r \to 0^+} \fint_{B_r^\pm(x,\nu_w)} |w(y) - w^\pm| \dd y =0, 
\end{equation} 
$x$ is called an (approximate) jump point of $w$.  The set of all jump points of $w$ is called the (approximate) jump set of $w$ and is denoted by $J_w$. Clearly $J_w \subset S_w$. However, the condition defining jump points is rather rigid: even for a general locally integrable function, the jump set is countably $\cH^{m-1}$-rectifiable \cite{DelNin}---that is, it can be covered up to a $\cH^{m-1}$-negligible set by a countable union of Lipschitz or, equivalently, $C^1$ graphs \cite[p.~80]{afp}. Recall that $\Gamma \subset \Omega$ is called a $C^1$ (resp.\ Lipschitz) graph if there exists a vector $\nu \in \bS^{m-1}$, a relatively open subset $U$ of a hyperplane in $\bR^m$ parallel to $\nu^\perp$ and a $C^1$ (resp. Lipschitz) function $\widetilde{\gamma} \colon U \to \bR$ such that image of the map $\gamma \colon U \to \Omega$ given by $\gamma(x') = x' + \widetilde{\gamma}(x') \nu$ coincides with $\Gamma$. A map of this form is called a \emph{graphical parametrization} of $\Gamma$.

We observe that $B_r^\pm(x,\nu_w)$ may be replaced with $Q_r^\pm(x,\nu_w)$ in \eqref{jump_def} without changing the definition. Moreover, if $x \in \Omega\setminus S_w$, then \eqref{jump_def} holds with $w^+ = w^- = z$ and any $\nu_w \in \bS^{m-1}$. Thus, \eqref{jump_def} defines a (multi)function $x \mapsto \{w^+(x), w^-(x)\}$ on $J_w \cup (\Omega \setminus S_w)$. For $x \in J_w$, the triple $(w^+(x),w^-(x),\nu_w(x))$ is defined uniquely up to a permutation of $(w^+, w^-)$ and a change of sign of $\nu_w$. In particular, the tensor product
$(w^+(x)-w^-(x))\otimes \nu_w(x)$ is uniquely defined for $x \in J_w$ (and for $x \in \Omega \setminus S_w$, where it vanishes). 

We recall the notion of Lebesgue points closely related to approximate continuity. If $\mu$ is a Radon measure on $\Omega$ and $w \in L^p_{loc}(\Omega,\mu)^n$, $p \in [1,\infty[$, we say that $x \in \Omega$ is a ($p$-)Lebesgue point of $w$ (with respect to $\mu$), if  
\[\lim_{r \to 0^+} \fint_{B_r(x)} |w(y) - w(x)|^p \dd \mu(y) =0.\]
It is known that $\mu$-almost every $x \in \Omega$ is a Lebesgue point for any given $w$ \cite[Section 1.7]{EvansGariepy}. We observe that every $p$-Lebesgue point is a $q$-Lebesgue point if $1 \leq q \leq p$; if $w \in L^\infty_{loc}(\Omega,\mu)^n$, the notion does not depend on $p$. We will use the notion of Lebesgue points in particular for functions in the space $L^p(\Gamma)^n$, with $\Gamma$ a $C^1$ graph contained in $\Omega$---we note that this space coincides with $L^p(\Omega, \cH^{m-1}\mres \Gamma)^n$. 

\subsection{Functions of bounded variation}\label{sec:BV}
Throughout the paper, we will consider convex functionals $\cE$ defined in $L^1_\loc (\Omega)^n$, $n\ge 1$, for $\Omega\subset\bR^m$ an open set. We will work with minimizers of $\cE$, which will be assumed to belong to $BV_\loc(\Omega)^n$. We recall that
\[
BV(\Omega)^n=\left\{ w\in L^1(\Omega)^n\,:\, TV(w)<\infty\right\}
\]
where the total variation $TV$ is defined by
\[
TV(w)=\sup\left\{ -\int w\Div \varphi\,\dd x\,:\,\varphi\in C_c^\infty(\Omega;\bR^m),\  |\varphi(x)|\le 1\ \text{for } x\in\Omega\right\}.
\]
It is easily checked (from Riesz's theorem) that $TV(w)$ is finite if and only if the distributional
derivative $Dw$ is a bounded Radon measure in $\Omega$, in which case
\[ TV(w) = \int_\Omega |Dw| = |Dw|(\Omega).\]
Then, one defines $BV_\loc(\Omega)^n= \bigcap_{A\subset\subset\Omega} BV(A)^n$, where the intersection
is on all open sets whose closure lies in $\Omega$. 

By the Federer--Vol'pert theorem \cite[Theorem 3.78]{afp}, if $w \in BV_{loc}(\Omega)^n$, the set $S_w$ is countably $\cH^{m-1}$-rectifiable and $\cH^{m-1}(S_w\setminus J_w) = 0$. In particular, the (multi)function $\{w^+, w^-\}$ is defined $\cH^{m-1}$-a.\,e.~in $\Omega$. Thus also the precise representative $\widetilde{w}$ of $w$ given by 
\[\widetilde{w} = (w^+ + w^-)/2\]
is defined up to $\cH^{m-1}$-null sets. In general, $w \in BV_{loc}(\Omega)^n$ admits one-sided traces on any oriented, countably $\cH^{m-1}$-rectifiable subset of $\Omega$, see \cite[Theorem 3.77]{afp}. Those traces coincide with $w^\pm$ $\cH^{m-1}$-a.\,e.~(up to permutation) \cite[Remark 3.79]{afp}. 

The Radon measure $Dw$ can be decomposed as:
\[
Dw = D^a w+ D^s w\,,\quad D^a w = \nabla w\, \cL^m \,,\quad D^s w = D^c w + (w^+-w^-)\otimes \nu_w\, \cH^{m-1}\mres J_w 
\]
where\begin{itemize}
    \item $D^a w$ is the absolutely continuous part of $Dw$, $D^s w$ the singular part, and $\nabla w \in L^1(\Omega)^{n\times m}$ is the Radon--Nikodym derivative of $Dw$ with respect to the Lebesgue measure $\cL^m$;
    \item $D^c w$ is the ``Cantor part'' of $Dw$, which is singular with respect to
    the Lebesgue measure and vanishes on sets of finite $(m-1)$-dimensional
    Hausdorff measure $\cH^{m-1}$;
    \item the last term $(w^+-w^-)\otimes \nu_w\, \cH^{m-1}\mres J_w$ is
    called the ``jump part'' of $Dw$.
\end{itemize}
The matrix $(D^c w/|D^c w|)(x)$ appearing in the polar decomposition of the Cantor part $D^c w
= (D^c w/|D^c w|)|D^c w|$  is known to have rank one for $|D^c w|$-a.\,e.~$x \in \Omega$ \cite{Alberti}, analogously to the jump part. We refer to~\cite{afp} for more details.

Similarly~\cite{TemamPlasticite}, $BD(\Omega)$ is defined as the space
of displacements $u\in L^1(\Omega)^m$ such that the symmetrized gradient
$Eu:=(Du+Du^T)/2$ is a bounded Radon measure, and one has:
\[
Eu = E^a w+ E^s w\,,\quad E^a w = e( w)\, \cL^m \,,\quad E^s w = E^c w + (w^+-w^-)\odot \nu_w\, \cH^{m-1}\mres J_w 
\]
with $e(w)\in L^1(\Omega)^{m\times m}$ the approximate symmetrized gradient,
$E^c w$ the Cantor part and $\odot$ the symmetrized tensor product ($a \odot b:= ((a_ib_j+a_j b_i)/2)_{i,j=1}^m$). Note that an analog of Alberti's rank one
theorem also holds in $BD$, see~\cite{DePhilippisRindler}.

\section{Setting and main result}\label{sec:main}

Let $\Omega \subset \bR^m$, $m\ge 1$ be an open set. We consider functionals $\cE \colon L^1_\loc(\Omega)^n \to [0, \infty]$ of form
\[\cE(w) = \cF(w-f) + \cR(w),  \] 
where $f \in L^1_\loc(\Omega)^n$. 
We assume the \emph{fidelity} $\cF \colon L^1_\loc(\Omega)^n \to [0, \infty]$ is given by 
\begin{equation*} 
	\cF(w) = \int_\Omega \psi(w) 
\end{equation*}
where $\psi \colon \bR^n \to [0, \infty[$ is convex. As for the \emph{regularizer} $\cR \colon L^1_\loc(\Omega)^n \to [0, \infty]$, in general we only assume that it is convex, without prescribing a particular structure. The regularizer contains prior information of the reconstructed image u, and will usually be defined as a convex integral of the distributional gradient $Dw$, possibly with an additional box constraint $w(x)\in\K$ a.\,e.\ for some closed convex set $\K \subset \bR^n$, enforced by prescribing $\cR(w) = \infty$ if $w$ does not satisfy it.   

Our aim in this paper is to provide an estimate on the jumps of minimizers of $\cE$, that is, functions $u \in L^1_\loc(\Omega)^n$ satisfying 
\begin{equation}\label{eq:mainpb}
\cE(u) = \inf \left\{\cE(w)\colon w\in L^1_\loc(\Omega)^n\right\}.
\end{equation}
In the case that $\psi$ is strictly convex, there is at most one $u$. However, without further assumptions, $u$ might not exist.  


Let $\varphi \in C_c^\infty(\Omega)^n$. For $w \in L^1_\loc(\Omega)^n$ and $\tau \in \bR$ with $|\tau|$ sufficiently small we put 
\[w^\varphi_\tau(x) = w(x + \tau \varphi(x)).\]
Suppose that $\cR(w) < \infty$. We say that $\cR$ is \emph{differentiable along inner variations} at $w$ if the limit 
\begin{equation}\label{inner_var_def}
	\lim_{\tau \to 0} \tfrac{1}{\tau}(\cR(w^\varphi_\tau) - \cR(w))
\end{equation}
exists for all $\varphi \in C_c^\infty(\Omega)^n$. 
In practice, we will only use \emph{directional} inner variations, where
$\varphi$ has the form $\widetilde{\varphi}\nu$ for $\nu\in \bS^{m-1}$ and
$\widetilde\varphi\in C_c^\infty(\Omega)$. Let us now state our main result. 

\begin{thm}\label{thethm}
	Suppose that $f \in BV_\loc(\Omega)^n$ and $u$ is minimizing in~\eqref{eq:mainpb} with $\cE(u) \neq \infty$. We assume 
	\begin{enumerate}[label=(H\arabic*)]
		\item \label{Hbv} $u \in BV_\loc(\Omega)^n$, 
\item \label{Hbdd} $u,f\in L^\infty_\loc(\Omega)^n$ or $D\psi$ is bounded, 
  \item \label{Hdiff} \newcounter{Hdiffc} \setcounter{Hdiffc}{\value{enumi}} $\cR$ is differentiable along directional inner variations at $u$. 
	\end{enumerate} 
If $\psi$ is $C^1$ and strictly convex, then $\cH^{m-1}(J_u\setminus J_f)=0$. If $\psi$ is $C^2$, then  
	\begin{equation} \label{main_ineq} 
		 (u^+ - u^-) \cdot A \, (u^+ - u^-) \leq (f^+ - f^-)\cdot A \, (u^+ - u^-) \qquad \cH^{m-1}\text{-a.\,e.~on } J_u,
	\end{equation} 
where 
\[A = \int_0^1 D^2 \psi (u^- - f^- + s(u^+ - f^+ - u^- + f^-)) \dd s.\] 
\end{thm}
In \eqref{main_ineq}, the selections of $u^\pm$ and $f^\pm$ are chosen in a mutually consistent manner. Technically they are determined by a chosen orientation of the sequence of $C^1$ graphs covering $J_u$ (see Section \ref{sec:BV}), but evidently \eqref{main_ineq} does not depend on this choice. It follows from \eqref{main_ineq} that 
\begin{equation*}  
	(u^+ - u^-) \cdot A \, (u^+ - u^-) \leq (f^+ - f^-)\cdot A \, (f^+ - f^-),
\end{equation*}
which can be translated into a bound on the size of jumps of $u$ in terms of $f$. In particular in the strongly convex, Lipschitz-gradient case 
\begin{equation}\label{sconv} 
 \lambda I \leq D^2 \psi \leq \Lambda I \quad \text{with } 0< \lambda \leq \Lambda, 
\end{equation}
we obtain
\[|u^+-u^-|\le \sqrt{\Lambda/\lambda} \,|f^+-f^-|.\] 
However, \eqref{main_ineq} also carries information about the jump direction in the value space $\bR^n$.    
The proof of Theorem~\ref{thethm} is postponed after the proofs of the following two lemmas. 
As mentioned in the introduction, a crucial tool here is the
idea of~\cite{Valkonen2015} to combine inner variations with the original function.
\begin{lemma} \label{lem:inner_diff}
  Let $u$ be the minimizer of $\cE$. Suppose that the limit \eqref{inner_var_def} exists for $w=u$. For $\vartheta \in [0,1]$, we denote 
  \[u^\varphi_{\vartheta, \tau} = \vartheta u_\tau^\varphi + (1 -\vartheta) u.\]
  Then, for $\vartheta \in [0,1]$,
  \begin{equation} \label{fid_ineq}
  \liminf_{\tau \to 0^+} \tfrac{1}{\tau}(\cF(u^\varphi_{\vartheta,\tau} - f) - \cF(u - f)) + \liminf_{\tau \to 0^+} \tfrac{1}{\tau}(\cF(u^\varphi_{\vartheta,-\tau} - f) - \cF(u - f)) \geq 0.
  \end{equation} 
\end{lemma} 
\begin{proof} 
  By minimality,
  \[0 \leq \liminf_{\tau \to 0^+}\tfrac{1}{\tau}(\cE(u^\varphi_{\vartheta, \pm \tau}) - \cE(u) ) = \liminf_{\tau \to 0^+}\tfrac{1}{\tau}(\cR(u^\varphi_{\vartheta, \pm\tau}) - \cR(u) +\cF(u^\varphi_{\vartheta,\pm\tau} - f) - \cF(u - f)).\]
  By our assumption, the function $R_\varphi\colon ]\!-\!\tau_0, \tau_0[\to [0, \infty[$ defined by 
  \[R_\varphi(\tau) = \cR(u^\varphi_\tau)\]
  for $\tau_0$ small enough is differentiable at $\tau = 0$. Thus, by convexity of $\cR$, 
  \[\tfrac{1}{\tau}(\cR(u^\varphi_{\vartheta, \pm \tau}) - \cR(u)) \leq \tfrac{\vartheta}{\tau}(\cR(u^\varphi_{\pm\tau}) - \cR(u)) \to \pm \vartheta R_\varphi'(0) \quad \text{as } \tau \to 0.\]
  Therefore, 
  \begin{equation*} 
    0 \leq \pm \vartheta R_\varphi '(0) + \liminf_{\tau \to 0^+} \tfrac{1}{\tau}(\cF(u^\varphi_{\vartheta,\pm\tau} - f) - \cF(u - f)). 
  \end{equation*} 
  We conclude by summing together the two obtained inequalities. 
\end{proof}
We note that the Lemma \ref{lem:inner_diff} is the only place in the proof of Theorem \ref{thethm} where we use convexity of $\cR$ (or differentiability of $\cR$ for that matter). In what follows, we will work directly with inequality \eqref{fid_ineq}. Thus, we could drop the convexity hypothesis altogether and instead assume explicitly that $\cR$ is differentiable along \emph{mixed variations} $u^\varphi_{\vartheta,\tau}$ (with fixed $\vartheta \in [0,1]$), leading directly to \eqref{fid_ineq}. In this way it might be possible to treat lower order perturbations of convex regularizers or the case of quasiconvex integrands, etc. 
However, we do not know simple and natural examples for which it is clear
that such differentiability holds---one may check in particular that the celebrated
Mumford--Shah functional~\cite{mum-shah} is not differentiable along
mixed variations near jump points.
In fact, also our proof of differentiability for concrete regularizers (Theorem \ref{thm:diff_ex}) uses the duality formula for convex integrands. Thus, we decided to keep the simpler assumption of differentiability along inner variations in the statement of our main result and leave the discussion of non-convex regularizers to a possible future work.  

We point out that differentiability along inner variations can roughly be seen as a relaxation of the \emph{double-Lipschitz comparability} of \cite{Valkonen2015}. The latter condition does not explicitly imply differentiability of $\cR$, instead imposing a certain quantitative upper bound on the quantity $\cR(w^\varphi_\tau) - 2\cR(w) + \cR(w^\varphi_{-\tau})$ for a general class of Lipschitz inner variations $\varphi$.    However, in the case that $\cR$ is convex, the difference quotients $\tfrac{1}{\tau}(\cR(w^\varphi_{\pm\tau}) - \cR(w))$ are non-decreasing functions of $\tau$, and their one-sided limits $R_\pm'$ as $\tau \to 0^+$ exist. Moreover, the sum of the two difference quotients is non-negative. On the other hand, by \cite[Lemma 6.5]{Valkonen2015}, double-Lipschitz comparability implies
\begin{equation}\label{dLc}
    0 \leq \tfrac{1}{\tau}(\cR(w^\varphi_\tau) - \cR(w)) + \tfrac{1}{\tau}(\cR(w^\varphi_{-\tau}) - \cR(w)) \leq C\tau \quad \text{for } \tau >0
\end{equation}
with $C>0$ for the particular (Lipschitz) directional variations $\varphi$ considered there. Thus, not only $R_+' = - R_-'$, but also we can write 
\begin{multline*}
    0 \leq \left|\tfrac{1}{\tau}(\cR(w^\varphi_\tau) - \cR(w)) - R_+'\right|+ \left|\tfrac{1}{\tau}(\cR(w^\varphi_{-\tau}) - \cR(w)) - R_-'\right| \\
    = \tfrac{1}{\tau}(\cR(w^\varphi_\tau) - \cR(w)) - R_+' + \tfrac{1}{\tau}(\cR(w^\varphi_{-\tau}) - \cR(w)) - R_-'
    \leq C\tau,
\end{multline*} 
i.\,e.\ not only the limit \eqref{inner_var_def} exists, but also one has an $O(\tau)$ estimate on the error term.

In the next lemma we investigate asymptotic behavior of the fidelity term under particular inner variations that push the values of $u$ from one side of a jump discontinuity to the other one (see Figure \ref{fig:push}). It can be seen as a variant of~\cite[Lemma~6.2]{Valkonen2015} and relies on a similar use of the properties of $BV$ functions
near a jump point. We note that minimality of $u$ is not used in the proof.  

Before stating the lemma, let us recall that given measurable spaces $X$, $Y$, a measure $\mu$ on $X$ and a measurable function $\Phi \colon X \to Y$, the formula $\Phi_\# \mu(A) := \mu(\Phi^{-1}(A))$ defines a measure $\Phi_\# \mu$ on $Y$ called the \emph{pushforward of $\Phi$ by $f$}. For any measurable function $f$ on $Y$ such that $f \circ \Phi$ is integrable w.\,r.\,t.\ $\mu$, $f$ is integrable w.\,r.\,t.\ $\Phi_\# \mu$ and we have 
\[\int f\, \dd \Phi_\# \mu = \int f \circ \Phi\, \dd \mu. \]

In the case that $\gamma \colon B^{m-1}_{r_0}(x_0, \nu_0) \to Q_{r_0}(x_0, \nu_0)$ is a graphical parametrization of a $C^1$ curve $\Gamma\subset \Omega$, we have by the area formula (where $\mathfrak{J}_\gamma = \sqrt{\det D\gamma^T D\gamma}$ is the Jacobian of $\gamma$) 
\begin{equation} \label{graph_push} 
\int_\Gamma f\, \dd \gamma_\# \cL^{m-1} = \int_{B^{m-1}_{r_0}(x_0, \nu_0)} \!\!\!\!\! f \circ \gamma \, \dd \cL^{m-1} = \int_\Gamma \frac{f}{\mathfrak{J}_\gamma\circ \gamma^{-1}}  \dd \cH^{m-1} = \int_\Gamma \frac{f}{\sqrt{1 + |D \widetilde{\gamma}\circ \gamma^{-1}|^2}}  \dd \cH^{m-1},
\end{equation} 
with $\widetilde{\gamma}$ such that $\gamma(x') = x' + \widetilde{\gamma}(x') \nu_0$. Thus we can write $\dd \gamma_\# \cL^{m-1} = \frac{1}{\sqrt{1 + |D \widetilde{\gamma}\circ\gamma^{-1}|^2}} \dd \cH^{m-1}\mres \Gamma$. 

\begin{figure}[h]
\begin{center}
  \includegraphics[width=.6\textwidth]{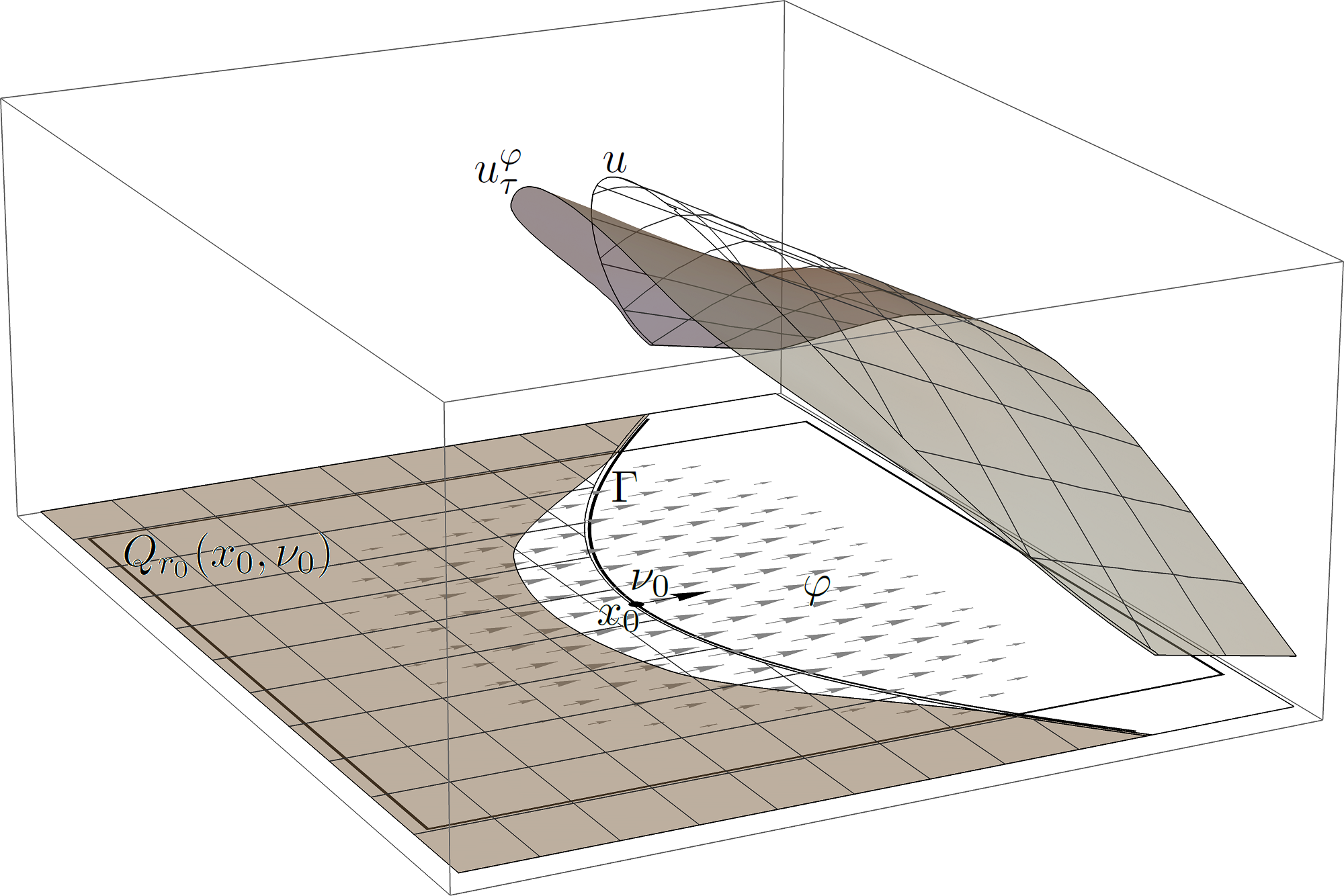}
\caption{Graphs of $u$ and $u_\tau^\varphi$ for given BV function $u$ with a jump discontinuity along a $C^1$ graph $\Gamma$, vector field $\varphi$ satisfying the assumptions of Lemma \ref{lem:fidelity}, and $\tau>0$.}\label{fig:push}
\end{center}
\end{figure}

\begin{lemma} \label{lem:fidelity}
  Assume \ref{Hbv} and \ref{Hbdd} hold. Let $\Gamma \subset Q_{r_0}(x_0, \nu_0) \subset \Omega$ be a $C^1$ graph admitting a graphical parametrization $\gamma \colon B^{m-1}_{r_0}(x_0, \nu_0) \to Q_{r_0}(x_0, \nu_0)$. Let $\varphi \in C_c^\infty(\Omega)^m$ be such that the support of $\varphi$ is contained in $Q_r(x_0, \nu_0)$, $0<r\leq r_0$, and $\varphi = \nu_0 \, \widetilde{\varphi}$,  $\widetilde{\varphi} \in C^\infty_c(\Omega)$. Moreover, assume that $0 \leq \widetilde{\varphi} \leq 1$ and restrictions of $\widetilde{\varphi}$ to lines parallel to $\nu_0$ attain their maxima on $\Gamma$. For simplicity we also assume that 
  \begin{equation}\label{phi_simple} 
    \nu_0 \cdot D\widetilde{\varphi}(x) = 0 \quad  
     \forall x \in  \Gamma + [-\varepsilon,\varepsilon]\nu_0, \quad\text{ for some } \varepsilon >0. 
  \end{equation}
  Then
  \begin{multline*} 
    \limsup_{\tau \to 0^+}\tfrac{1}{\tau}  (\cF(u^\varphi_{\vartheta, \tau} - f) - \cF(u - f)) \\ \leq \int_{\Gamma} \widetilde{\varphi} \left( \psi(\vartheta u^+ + (1- \vartheta) u^- - f^-) - \psi(u^- - f^-)\right)\dd \gamma_{\#} \cL^{m-1}  + C \vartheta \left| \nu_0 \cdot \D u\right|(Q_r(x_0,\nu_0)\setminus \Gamma), 
  \end{multline*} 
  \begin{multline*} 
    \limsup_{\tau \to 0^+}\tfrac{1}{\tau}  (\cF(u^\varphi_{\vartheta, -\tau} - f) - \cF(u - f)) \\ \leq \int_{\Gamma} \widetilde{\varphi} \left( \psi(\vartheta u^- + (1- \vartheta) u^+ - f^+) - \psi(u^+ - f^+)\right)\dd \gamma_{\#} \cL^{m-1}  + C \vartheta \left| \nu_0 \cdot \D u\right|(Q_r(x_0,\nu_0)\setminus \Gamma),  
  \end{multline*} 
  where $\gamma_{\#} \cL^{m-1}$ denotes the pushforward of $\cL^{m-1}$ by $\gamma$ and 
  \begin{equation} \label{C_def}
      C := \sup \left\{ |D \psi(\xi)|\colon \xi \in \bR^n, \ |\xi| \leq \|u\|_{L^\infty(Q_{r_0}(x_0,\nu_0))^n} + \|f\|_{L^\infty(Q_{r_0}(x_0,\nu_0))^n} \right\}
  \end{equation} 
  is finite by virtue of \ref{Hbdd}. 
\end{lemma} 

\begin{proof} 
  By an isometric change of coordinates, we will assume that $\nu_0 = e_m$, $x_0 = 0$ and denote $x = (x', x_m)$, $Q_r(x_0,\nu_0) = Q_r$, $B^{m-1}_r(x_0,\nu_0) = B^{m-1}_r$, $\gamma(x') = (x', \widetilde{\gamma}(x'))$ for $x' \in B^{m-1}_r$, so that $\Gamma = \{\gamma(x')\colon x' \in B^{m-1}_r\}$. By our assumption, we have 
  \[x + \tau \varphi(x) = (x', x_m + \tau \widetilde{\varphi}(x)).\]
  We will prove the first part of the assertion. The proof of the second one is the same. We take $\tau > 0$ small enough so that the maps $x \mapsto x \pm \tau \varphi(x)$ are diffeomorphisms. In particular, $x_m \mapsto x_m \pm \tau \widetilde{\varphi}(x', x_m)$ are diffeomorphisms for every $x' \in B^{m-1}_r$. By \eqref{phi_simple}, we can also assume
  \begin{equation}\label{phi_const} \widetilde{\varphi}(x', x_m) = \widetilde{\varphi}(\gamma(x')) \quad \text{whenever } - \tau \widetilde{\varphi}(x', x_m) \leq x_m - \widetilde{\gamma}(x') \leq \tau \widetilde{\varphi}(x', x_m). 
  \end{equation} 
  We rewrite   
  \begin{multline*} 
    \cF(u^\varphi_{\vartheta, \tau} - f) - \cF(u - f) = \int_{B^{m-1}_r}\int_{\widetilde{\gamma}(x') - \tau \widetilde{\varphi}(\gamma(x'))}^{\widetilde{\gamma}(x')} \psi(u^\varphi_{\vartheta, \tau} - f) - \psi(u - f)\, \dd x_m\, \dd x' \\
    + \int_{B^{m-1}_r}\int_{]\!-\!r,r[\setminus[\widetilde{\gamma}(x') - \tau \widetilde{\varphi}(\gamma(x')),\widetilde{\gamma}(x')]} \psi(u^\varphi_{\vartheta, \tau} - f) - \psi(u - f)\, \dd x_m\, \dd x' =: I^1_{\vartheta, \tau} + I^2_{\vartheta, \tau}.
  \end{multline*} 
  
  We first estimate $I^2_{\vartheta, \tau}$. Identifying $u$ with its precise representative, the slicing properties of $BV$ functions~\cite[\S 3.11, Theorem 3.107]{afp} ensure that for $\cL^{m-1}$-a.\,e.\ $x' \in B^{m-1}_r$ the function $u_{x'} \colon x_m\mapsto u(x',x_m)$ is in $BV(]\!-r,r[)^n$, and we can write for 
  $\cL^1$-a.\,e.~$x_m \in ]\!- r, r[$:
  \begin{multline*} u^\varphi_{\vartheta, \tau}(x',x_m) - u(x', x_m) = \vartheta \left(u^\varphi_\tau(x',x_m) -  u(x', x_m)\right)  \\
  =\vartheta\left( u_{x'}(x_m + \tau \widetilde{\varphi}(x', x_m)) - u_{x'}(x_m)\right) =\vartheta\, Du_{x'} \left(\left]x_m,x_m + \tau \widetilde{\varphi}(x', x_m)\right[\right).
  \end{multline*} 
  Therefore
\begin{multline}\label{psi_diff_est}
\left|\psi(u^\varphi_{\vartheta, \tau} - f) - \psi(u - f)\right|(x', x_m) 
      = \left|\int_0^1 D\psi(u+t(u^\varphi_{\vartheta, \tau}-u) - f) \dd t \,\cdot(u^\varphi_{\vartheta, \tau}-u)\right|(x', x_m)  \\
         \leq C \vartheta \left|Du_{x'} \right|\left(\left]x_m,x_m + \tau \widetilde{\varphi}(x', x_m)\right[\right) 
         \leq C \vartheta \left|Du_{x'}\right|\left(\left]x_m,\min\{r,x_m + \tau \widetilde{\varphi}(\gamma(x'))\}\right[\right), \end{multline}
      where $C$ is defined in \eqref{C_def} and we have used that $\widetilde{\varphi}(x',\cdot)$ is maximal at $\widetilde{\gamma}(x')$.
	Thus, for $\cL^{m-1}$-a.\,e.\ $x' \in B^{m-1}_r$,
 \begin{multline*} 
		\frac{1}{\tau}\int_{]\!-\!r,r[\setminus[\widetilde{\gamma}(x') - \tau \widetilde{\varphi}(\gamma(x')),\widetilde{\gamma}(x')]}\left|\psi(u^\varphi_{\vartheta, \tau} - f) - \psi(u - f)\right| \dd x_m
  \\  \leq C \frac{\vartheta}{\tau} \int\int \chi_{]\!-\!r,r[\setminus[\widetilde{\gamma}(x') - \tau \widetilde{\varphi}(\gamma(x')),\widetilde{\gamma}(x')]}(x_m) \chi_{]x_m,\min\{r,x_m + \tau \widetilde{\varphi}(\gamma(x'))\}[}(t) \,\dd \big|Du_{x'}\big|(t)\dd x_m.
	\end{multline*}
  The integrand is non-zero only when $x_m < t < \min\{r,x_m + \tau \widetilde{\varphi}(\gamma(x'))\}$ and either $-r < x_m < \widetilde{\gamma}(x')-\tau\widetilde{\varphi}(\gamma(x'))$
 or $\widetilde{\gamma}(x') < x_m\le r$, in particular one has $t\in ]\!-r,r[\setminus \{\widetilde{\gamma}(x')\}$ and
 $t-\tau\widetilde{\varphi}(\gamma(x'))\le x_m\le t$. Using Fubini's theorem, we
 deduce that this expression is bounded by
 \[ C\vartheta\widetilde\varphi(\gamma(x')) |Du_{x'}|(]\!-\!r,r[\setminus \{\widetilde{\gamma}(x')\}).
 \]
	Appealing to the slicing formula from \cite[Theorem 3.107]{afp}, 
	\begin{equation*}\frac{1}{\tau}\int_{B^{m-1}_r}\int_{]\!-\!r,r[\setminus[\widetilde{\gamma}(x') - \tau \widetilde{\varphi}(\gamma(x')),\widetilde{\gamma}(x')]}
    \hspace{-1ex}
 \psi(u^\varphi_{\vartheta, \tau} - f) - \psi(u - f)\, \dd x_m\, \dd x'  \leq C \vartheta |\nu_0\cdot\D u|(Q_r\setminus \Gamma).\end{equation*}

It remains to pass to the limit $\tau \to 0^+$ in $I^1_{\vartheta, \tau}$. By \cite[Theorem 3.108]{afp}, we have $u(x', s) \to u^\pm(\gamma(x'))$ as $s \to \gamma(x')^\pm$ for $\cL^{m-1}$-a.\,e.\ $x' \in B^{m-1}_r$, where we recall that we have chosen $u^-$, $f^-$ (resp.\ $u^+$, $f^+$) to be the approximate limits corresponding to traces of $u$, $f$ along $\Gamma$ ``from below'' (resp.\ ``from above''), consistently with the choice of $u^\pm(x_0)$ given by $\nu_0$. Thus, using \eqref{phi_const} to deal with $u^\varphi_{\vartheta, \tau}$, for $\cL^{m-1}$-a.\,e.\ $x' \in B^{m-1}_r$ we have as $\tau\to 0$:
\begin{equation*}
    \frac{1}{\tau}\int_{\widetilde{\gamma}(x') - \tau \widetilde{\varphi}(\gamma(x'))}^{\widetilde{\gamma}(x')}\, \psi(u^\varphi_{\vartheta, \tau} - f)\, \dd x_m 
    \to \widetilde{\varphi}\, \psi(\vartheta u^+ + (1-\vartheta) u^- - f^-)\bigg|_{\gamma(x')},
  \end{equation*}
  \begin{equation*}
    \frac{1}{\tau}\int_{\widetilde{\gamma}(x') - \tau \widetilde{\varphi}(\gamma(x'))}^{\widetilde{\gamma}(x')} \psi(u - f)\, \dd x_m \to \widetilde{\varphi}\, \psi(u^- - f^-)\bigg|_{\gamma(x')}. 
  \end{equation*} 
  Hence, the integrand of $I^1_{\vartheta, \tau}$ converges $\cL^{m-1}$-a.\,e. Using \eqref{psi_diff_est} and reasoning as before, we obtain a rough estimate 
  \begin{multline}
  \frac{1}{\tau}\left|\int_{\widetilde{\gamma}(x') - \tau \widetilde{\varphi}(\gamma(x'))}^{\widetilde{\gamma}(x')} \psi(u^\varphi_{\vartheta, \tau} - f) - \psi(u - f)\, \dd x_m\right| \\ \leq
  C \frac{\vartheta}{\tau} \int\int \chi_{[\widetilde{\gamma}(x') - \tau \widetilde{\varphi}(\gamma(x')),\widetilde{\gamma}(x')]}(x_m) \chi_{]x_m,\min\{r,x_m + \tau \widetilde{\varphi}(\gamma(x'))\}[}(t) \,\dd \big|Du_{x'}\big|(t)\dd x_m \\ \leq C\vartheta\widetilde\varphi(\gamma(x')) |Du_{x'}|(]\!-\!r,r[)
  \end{multline}
  for $x' \in B^{m-1}_r$. Since the r.h.s.\ is an integrable function (again by \cite[Theorem 3.107]{afp}), we can invoke the dominated convergence theorem to compute the limit $\lim_{\tau\to 0}I^1_{\vartheta, \tau}$. Recalling the first equality in \eqref{graph_push}, we rephrase the resulting integral in terms of the pushforward and conclude the proof of the Lemma. 
\end{proof} 
	
\begin{proof}[Proof of Theorem \ref{thethm}]
Let $(\Gamma_i)_{i=1}^\infty$ be a sequence of $C^1$ graphs that covers $J_u$ up to a $\cH^{m-1}$-null set. Let us fix an index $i$. By \cite[eq.~(2.41) on p.~79]{afp}, for $\cH^{m-1}$-a.\,e.~$x_0$ in $J_u\cap\Gamma_i$
\begin{equation}\label{valk_density}
  \Theta^{m-1}(|\D u|\mres (\Omega\setminus \Gamma_i),x_0) = 0.
\end{equation}
We choose such an $x_0$, and assume in addition that $x_0$ is a Lebesgue point for $u^{\pm}$ and $f^{\pm}$ with respect to $\cH^{m-1} \mres \Gamma_i$. We take $\nu_0 = \nu_u(x_0)$. For $r_0>0$ small enough $Q_{r_0}(x_0, \nu_0) \subset \Omega$ and $\Gamma := \Gamma_i \cap Q_{r_0}(x_0, \nu_0)$ has a graphical parametrization $\gamma \colon B^{m-1}_{r_0}(x_0, \nu_0) \to Q_{r_0}(x_0, \nu_0)$. Moreover, since $\Gamma$ is tangent to $x_0 +\nu_0^\perp$ at $x_0$, possibly decreasing $r_0$, we can assume that $\gamma(x') \in x' + [-r/2, r/2]\nu_0$ for $x' \in B^{m-1}_r(x_0, \nu_0)$, $0<r < r_0$ and construct a sequence of $\widetilde{\varphi}$ satisfying the assumptions of Lemma \ref{lem:fidelity} that converges to $1$ on $Q_r(x_0, \nu_0)\cap \Gamma$. 
	Then, by Lemmas~\ref{lem:inner_diff} and \ref{lem:fidelity}, 
	\begin{multline*}
		0 \leq \int_{Q_r(x_0,\nu_0) \cap \Gamma}\!\!\!
  \psi(\vartheta u^+ + (1- \vartheta) u^- - f^-) - \psi(u^- - f^-) \,\dd \gamma_{\#} \cL^{m-1} \\ + \int_{Q_r(x_0,\nu_0) \cap \Gamma}\!\!\! \psi(\vartheta u^- + (1- \vartheta) u^+ - f^+) - \psi(u^+ - f^+)\, \dd \gamma_{\#} \cL^{m-1}  + 2 C \vartheta \left| \nu_0 \cdot \D u\right|(Q_r(x_0,\nu_0)\setminus \Gamma). 
	\end{multline*}
	Dividing 
 by 
 $\cL^{m-1}(B^{m-1}_r)\sim r^{m-1}$ and passing to the limit $r \to 0^+$, we get 
	\[0 \leq \psi(\vartheta u^+ + (1- \vartheta) u^- - f^-) - \psi(u^- - f^-) +  \psi(\vartheta u^- + (1- \vartheta) u^+ - f^+) - \psi(u^+ - f^+) \]
	at $x_0$ by \eqref{valk_density}. By convexity of $\psi$, 
	\[\psi(\vartheta u^+ + (1- \vartheta) u^- - f^-) - \psi(u^- - f^-) \leq \vartheta D \psi(\vartheta u^+ + (1- \vartheta) u^- - f^-) \cdot (u^+ - u^-),\]
	\[\psi(\vartheta u^- + (1- \vartheta) u^+ - f^+) - \psi(u^+ - f^+) \leq \vartheta D \psi(\vartheta u^- + (1- \vartheta) u^+ - f^+) \cdot (u^- - u^+).\]
	Summing up, 
	\begin{equation*}
          0 \leq \vartheta \left(D \psi(\vartheta u^+ + (1- \vartheta) u^- - f^-) - D \psi(\vartheta u^- + (1- \vartheta) u^+ - f^+)\right) \cdot (u^+ - u^-) .
	\end{equation*} 
	Dividing the inequality by $\vartheta$ and letting $\vartheta \to 0^+$ yields
	\begin{equation} \label{nice_ineq}
		0 \leq \left(D \psi(u^- - f^-) - D \psi(u^+ - f^+)\right) \cdot (u^+ - u^-).
	\end{equation}
        If $f^+=f^-$ and $\psi$ is strictly convex, then
        $D\psi$ is strictly monotone and  we get
        a contradiction unless $u^+=u^-$: this shows that $\cH^{m-1}(J_u\setminus J_f)=0$.
        On the other hand, if $\psi \in C^2$, applying the fundamental theorem of calculus to the function 
        \[s \mapsto D \psi \left(u^- - f^- + s(u^+ - f^+ - u^- + f^-)\right) \cdot (u^+ - u^-),\]
        and using \eqref{nice_ineq} we obtain \eqref{main_ineq}. 

\end{proof} 
    \begin{remark}\label{remthethm}
    Assumption~\textit{\ref{Hbv}} will hold as soon as $\cR$ is coercive
    in $BV(\Omega)^n$; the next section~\ref{sec:DiffReg} discusses
    a class of such regularizers for which \textit{\ref{Hdiff}} holds as well.
    As for \textit{\ref{Hbdd}}, it is trivially satisfied if $\psi$ is globally
    Lipschitz. Otherwise, we need to ensure that $u$ and $f$ are locally bounded. Assuming that $f\in L^\infty(\Omega)^n$, a minimizer of $\cE$ will be bounded
    if $n=1$ and $\psi$ is coercive, while if $n\ge 1$ this requires
    various types of assumptions on $\psi$ and $\cR$, see Section~\ref{sec:Bounded}. Alternatively, we can enforce it by considering a constrained minimization problem, which amounts to replacing $\cR$ with 
    \[\cR_{\K} :=\cR + \iota_{\{w \in L^1_\loc(\Omega)^n \colon w(x) \in \K \text{ for  a.\,e.\ } x \in \Omega\}}
    \]
    where $\K$ is bounded, closed and convex (here $\iota_U$ denotes the convex-analytic characteristic function of $U \subset L^1_\loc(\Omega)^n$, which is 0 in $U$ and $\infty$ outside). We observe that if $\cR$ satisfies \textit{\ref{Hdiff}}, then $\cR_{\K}$ satisfies it as well. Moreover, if $\cR$ is coercive in $BV(\Omega)^n$ (that is, $w \in BV(\Omega)^n$ whenever $\cR(w) < \infty$), then $\cR_{\K}$ is as well.  
    \end{remark}
    
    \section{Differentiable regularizers} \label{sec:DiffReg}
\subsection{General form}
    In this section we give some examples of functionals $\cR$ whose domain  contains $BV(\Omega)^n$ and that are differentiable along inner variations, in particular assumptions \textit{\ref{Hbv}} and \textit{\ref{Hdiff}} in Theorem \ref{thethm} are satisfied whenever $\cR(u) < \infty$. We assume for simplicity that $\Omega$ is bounded. Let $\varrho \colon \bR^{n \times m} \to [0, \infty[$ be convex and satisfy 
    \begin{equation} \label{lin_growth}
    \varrho(\xi) \leq C(1 + |\xi|) \quad \text{for } \xi \in \bR^{n \times m}
    \end{equation} 
    with $C >0$. For a vector Radon measure $\mu$ on $\Omega$ we denote by $\mu^{a}$ its density with respect to the Lebesgue measure on $\Omega$ and by $\mu^s$ its Lebesgue singular part, i.\,e.\ 
    \[\dd \mu = \mu^{a} \dd \cL^m + \dd \mu^s, \quad  \cL^m \perp \mu^s.\] 
    Moreover, for $\xi \in \bR^{n \times m}$ we denote 
    \[\varrho^\infty(\xi) = \lim_{t \to \infty} \tfrac{1}{t} \varrho(t \xi). \]
    Then 
    \[\dd \varrho(\mu) = \varrho(\mu^{a}) \dd \cL^m + \varrho^\infty\left(\tfrac{\mu^s}{|\mu^s|}\right) \dd |\mu^s|\]
    defines a Radon measure $\varrho(\mu)$ and 
    \begin{equation}\label{R_rho} 
    \cR(w) = \left\{\begin{array}{ll} \varrho(Dw)(\Omega) & \text{if } w \in BV(\Omega)^n,\\ 
    \infty & \text{otherwise}\end{array}\right. 
    \end{equation} 
    defines a convex functional on $L^1_\loc(\Omega)$~\cite{DemengelTemam}. At least when $\varrho$ satisfies a lower bound $\varrho(A) \geq C'|A|$ with $C' > 0$, $\cR$ is lower semicontinuous and $\cE$ attains its minimum at $u \in BV(\Omega)$. 
    \begin{thm}\label{thm:diff_ex}
    Suppose that 
    \begin{enumerate}[label=(D\arabic*)] 
    \item \label{Diffac} the function $\tau \mapsto \varrho\left(A(I + \tau B)\right)$
    is differentiable at $\tau = 0$ for any $A \in \bR^{n\times m}$ and any $B \in \bR^{m \times m}$,  
    \item \label{Diffs} the function $\tau \mapsto \varrho^\infty \left(A(I + \tau B)\right) $
    is differentiable at $\tau = 0$ for any $A \in \bR^{n\times m}$ of rank $1$ and any $B \in \bR^{m \times m}$.
    \end{enumerate} 
    Then $\cR$ is differentiable along inner variations at any $w \in \cD$.
    \end{thm} 
    The proof of Theorem \ref{thm:diff_ex} follows along the lines of \cite[Chapter 10]{Giusti} (where the case of total variation is considered). The main point is the following change of variables formula. 
    \begin{lemma}\label{lem:change_var} 
    Let $F \colon \Omega \to \Omega$ be a diffeomorphism ($C^1$ up to the boundary) and let $w \in BV(\Omega)^n$. Then
    \begin{equation}\label{rep_formula} \cR(w \circ F) = \int_\Omega |\det D (F^{-1})|\, \dd \varrho(Dw\, DF\circ F^{-1}).  
    \end{equation}
    \end{lemma} 

    \begin{proof} 
    We will use the dual representation of $\cR$ as given in \cite[Lemma 1.1]{DemengelTemam}, according to which 
    \begin{equation}\label{dual_rep} \int_\Omega \varphi\, \dd \varrho(\mu) = \sup_{h \in D_\varrho} \int_\Omega \varphi h \cdot \dd \mu - \int_\Omega \varphi \varrho^*(h) \dd \cL^m ,
    \end{equation} 
    for $\mu\in M(\Omega)^{n \times m}$ (the space of matrix-valued finite Radon measures; recall that $M(\Omega) \cong C_0(\Omega)^*$) and $\varphi \in C(\overline{\Omega})$ with $\varphi \geq 0$, where
    \[ D_\varrho = \{ h \in C_c(\Omega)^{n \times m} \colon \varrho^*(h) \in L^1(\Omega)\}.\]
    Note that here and in \eqref{dual_rep}, the notation $\varrho^*(h)$ stands for the composition of $\varrho^*$ with $h$, consistently with the rest of the paper. We reserve the  formal symbol $\circ$ for ``inner'' compositions with maps $\Omega \to \Omega$. 
    
    We recall \cite[Thm.~1.17]{Giusti}\cite[Thm.~3.9]{afp} that for each $w \in BV(\Omega)^n$ there exists a sequence $(w_k) \subset C^\infty(\Omega)^n$ such that $w_k \to w$ in $L^1(\Omega)^n$ and  
    \[\int_{\Omega} |D w_k|\dd \cL^m \to |Dw|(\Omega), \]
    in particular $D w_k \weaklystar D w$ in $M(\Omega)^{n \times m}$. It is easy to check that $w_k \circ F \to w \circ F$ in $L^1(\Omega)^n$. We also have 
    \[ \int_\Omega |D (w_k \circ F)| \leq \int_\Omega |D w_k \circ F| \, |DF| = \int_\Omega |D w_k| \, |DF \circ F^{-1}| |\det D(F^{-1})|, \]
    in particular $w_k \circ F$ is bounded in $BV(\Omega)^n$, whence $w \circ F \in BV(\Omega)^n$ and 
    \begin{equation} \label{weak_comp} 
    D (w_k \circ F) \weaklystar D (w \circ F) \quad \text{in } M(\Omega)^{n \times m}. 
    \end{equation}
    For any $h \in C_c(\Omega)^{n \times m}$, we can calculate 
    \begin{equation*} \label{change_calc}
    \int_\Omega h \circ F \cdot D (w_k \circ F)  = \int_\Omega h \circ F \cdot Dw_k \circ F \, DF = \int_\Omega h \cdot Dw_k \, DF \circ F^{-1} |\det D(F^{-1})|.     
    \end{equation*} 
    Using \eqref{weak_comp}, we pass to the limit $k\to \infty$ obtaining 
    \begin{equation} \label{change_test}  
    \int_\Omega h \circ F \cdot \dd D (w \circ F) = \int_\Omega |\det D(F^{-1})| h \cdot \dd Dw \, DF \circ F^{-1} . 
    \end{equation} 
    Note that 
    \begin{equation} \label{change_star} \int_\Omega \varrho^*(h \circ F) = \int_\Omega |\det D(F^{-1})| \varrho^*(h) , 
    \end{equation} 
    in particular $h \in D_\varrho$ iff $h \circ F \in D_\varrho$. Thus, using the dual representation formula \eqref{dual_rep} in conjunction with \eqref{change_test} and \eqref{change_star}, we obtain 
    \begin{multline*}
    \cR(w \circ F) = \varrho(D(w \circ F))  \\= \sup_{h \in D_\varrho} \int_\Omega h \cdot \dd D(w \circ F) - \int_\Omega \varrho^*(h) \dd \cL^m  = 
    \sup_{h \in D_\varrho} \int_\Omega h\circ F \cdot \dd D(w \circ F) - \int_\Omega \varrho^*(h \circ F) \dd \cL^m \\ =  \sup_{h \in D_\varrho}\int_\Omega |\det D(F^{-1})| h \cdot \dd Dw \, DF \circ F^{-1} - \int_\Omega |\det D(F^{-1})| \varrho^*(h)\dd \cL^m
    \\=  
     \int_\Omega |\det D (F^{-1})|\, \dd \varrho(Dw\, DF\circ F^{-1}) .
    \end{multline*} 
    
        \end{proof}

    \begin{proof}[Proof of Theorem \ref{thm:diff_ex}]
    We take $F_\tau(x) = x + \tau \varphi(x)$ with $\varphi \in C^\infty_c(\Omega)^m$, $x \in \Omega$ and $\tau$ small enough, so that $F_\tau$ is an orientation-preserving diffeomorphism. By the representation formula \eqref{rep_formula} and recalling the definition of the measure $\varrho(Dw\, DF_\tau\circ F_\tau^{-1})$,  
    \begin{multline} \label{comp_expl}
    \cR(w \circ F_\tau) = \int_\Omega \det D (F_\tau^{-1})\, \varrho(Dw^{a}\, DF_\tau\circ F_\tau^{-1})\, \dd \cL^m \\ + \int_\Omega \det D (F_\tau^{-1})\, \varrho^\infty\left(\tfrac{Dw^s}{|Dw^s|}\, DF_\tau\circ F_\tau^{-1}\right) \dd |Dw^s|.
    \end{multline}
    Note that $DF_\tau\circ F_\tau^{-1} = I + \tau D \varphi \circ F_\tau^{-1}$.  Denoting $\varrho_x(\tau) = \varrho(Dw^{a}(x)\, DF_\tau\circ F_\tau^{-1}(x))$, we have for $\tau \neq 0$  
    \begin{multline*}\tfrac{1}{\tau}(\varrho_x(\tau) - \varrho_x(0)) = \tfrac{1}{\tau}(\varrho(Dw^{a}(x) (I + \tau D \varphi(x))) - \varrho(Dw^{a}(x)))\\ + \tfrac{1}{\tau}(\varrho(Dw^{a}(x) (I + \tau D \varphi\circ F_\tau^{-1}(x))) - \varrho(Dw^{a}(x) (I + \tau D \varphi(x))))
    \end{multline*} 
    Using assumption~\textit{\ref{Diffac}}, the global Lipschitz continuity of $\varrho$ and the continuity of $\tau \mapsto F_\tau^{-1}$, we deduce that $\varrho_x$ is differentiable at $\tau = 0$ for $\cL^m$-a.\,e.\ $x \in \Omega$ and there exists $C_{\varrho, \varphi} > 0$ such that
    \[\tfrac{1}{\tau}|\varrho_x(\tau) - \varrho_x(0)| \leq C_{\varrho, \varphi} |Dw^{a}(x)|. \]
    Thus, by dominated convergence, the first integral in \eqref{comp_expl} is differentiable at $\tau = 0$. Similarly, appealing to Alberti's rank-one theorem~\cite[Thm.~3.94]{afp} \cite{Alberti} and assumption\textit{\ref{Diffs}}, we show that the second integral in \eqref{comp_expl} is differentiable at $\tau = 0$.  
    \end{proof} 
    
    \subsection{Examples} \label{sec:examples}
    A simple way to ensure that conditions \textit{\ref{Diffac}} and \textit{\ref{Diffs}} of Theorem \ref{thm:diff_ex} are satisfied is to assume that $\varrho$ and $\varrho^\infty$ are differentiable everywhere except at $0$. In particular, \emph{vectorial} (such as defined by the Frobenius norm) and \emph{anisotropic (vectorial) total variations} given by any norm on $\bR^{n \times m}$ that is differentiable outside $0$, such as $\ell^{p,q}$, $p, q \in ]1, \infty[$, are differentiable along inner variations. The same is not true in the limiting cases $1, \infty$. In fact there are known examples where the assertion of Theorem \ref{thethm} fails in these cases, see~\cite{LasMolMuc2017}.

    A more striking example is given by the Nuclear or Trace norm:
    \begin{equation}\label{eq:nuclear}
        \varrho(A) = \text{Trace}((A A^T)^{\frac{1}{2}}),
    \end{equation}
    which is the sum of the singular values. The function:
    \[ \tau\mapsto A(I+\tau B)(I+\tau B)^T A^T\]
    is a one-parameter analytic function with values symmetric  $n\times n$ 
    matrices, so that by~\cite[p.~31]{Rellich}, the squared singular values
    can  also be described by analytic functions (and in particular, in a neighborhood
    of $\tau=0$, the positive eigenvalues stay positive). In addition,
    the kernel of $(I+\tau B)^T A^T$ is the same as the kernel of $A^T$ for
    small $\tau$, so that the number of non-zero singular values remains
    constant near $0$. We deduce that the sum of the singular values is
    also an analytic function near $0$ so that $\varrho$ satisfies
    the assumptions of Theorem~\ref{thm:diff_ex}. 

    In addition, we also deduce that any convex and one-homogeneous, differentiable and symmetric function
    of the singular values will enjoy the same properties, such as the
    $p$-Schatten norms for $p\in ]1,\infty[$. (The $p$-Schatten norm is the $\ell_p$ norm of the singular values.)

    {
    More generally for $A\in\bR^{n\times m}$, letting $p=\min\{m,n\}$
    we denote $\sigma(A)=(\sigma_1(A),\dots,\sigma_p(A))$ the ordered singular values of $A$. 
    We give a simple proof of the following standard result on ``unitary invariant'' convex functions of matrices, as a corollary of von Neumann's inequality.
    \begin{prop} 
    Let $\widetilde{h}$ be a 
    proper, lower semicontinuous, extended real-valued 
    convex function on $\bR^p_+$, non-decreasing with respect to each coordinate. 
    Then 
    \[h(A):=\widetilde h(\sigma(A))=\widetilde h(\sigma_1(A),\dots,\sigma_p(A))\]
    defines a convex function on $\bR^{n \times m}.$
    \end{prop} 
    \begin{proof} 
    We first observe that if we extend $\widetilde{h}$ to the whole $\bR^p$ 
    as an even function with respect to each coordinate:
    \begin{equation}\label{even_each}
    \widetilde{h}(s_1, \ldots, s_p) = \widetilde{h}(|s_1|, \ldots , |s_p|) \quad \text{for } (s_1, \ldots, s_p) \in \bR^p,
    \end{equation} 
    then $\widetilde{h}$ is convex on $\bR^p$.
    For any extended-real valued function on $\bR^p$ satisfying \eqref{even_each}, the same holds for its convex conjugate (and biconjugate). Moreover, if $t \in \bR^p_+$, then 
    \[\widetilde{h}^{*}(t) = \sup_{s \in \bR^p_+} t \cdot s - \widetilde{h}(s).\]
    It then follows from von~Neumann's inequality~\cite{vonNeumann,Mirsky} $\Tr AB^T \le \sum_{i=1}^p \sigma_i(A)\sigma_i(B)$ that
    \[
    h^*(B) = \sup_{A\in\bR^{n\times m}} \Tr A B^T - h(A) \leq \sup_{s\in\bR_+^p} s\cdot \sigma(B) -\widetilde h(s) = \widetilde{h}^*(\sigma(B)),
    \]
    while the opposite inequality is obvious, choosing $A$ a matrix with 
    singular values $s \in \bR_+^p$ and the same singular vectors as $B$. 
    By the same reasoning and the Fenchel-Moreau theorem, we deduce
    \[
    h^{**}(A) = \sup_{B\in\bR^{m\times n}} \Tr A B^T - h^*(B) = \sup_{s\in\bR_+^p} \sigma(A)\cdot s -\widetilde{h}^*(s) = \widetilde{h}^{**}(\sigma(A)) = \widetilde{h}(\sigma(A)) = h(A). 
    \]
    This shows that $h$ is convex.
    \end{proof}
    }
    We remark that if in addition $\widetilde{h}$ is smooth and a symmetric function of its arguments, then the discussion
    above for the Nuclear norm applies and ${h}$ is differentiable along inner variations. An interesting example is the following: we consider
    \[
    \varrho(A) = \log \sum_{i=1}^p \exp(\sigma_i(A)).
    \]
    We claim that $\varrho$ satisfies the assumptions of Theorem~\ref{thm:diff_ex}. Indeed, on the one hand, $\varrho(A)$ is smooth
    and satisfies \textit{\ref{Diffac}}. On the other hand, one readily checks that
    \[
    \varrho^\infty(A)=\lim_{t\to +\infty} \tfrac{1}{t}\varrho(tA) = \max\{\sigma_1(A),\dots,\sigma_p(A)\}.
    \]
    Therefore $\varrho^\infty$ is the Spectral (or Operator) norm, which does not satisfy \textit{\ref{Diffac}}, yet satisfies \textit{\textit{\ref{Diffs}}} since it coincides with the Frobenius (as well as Nuclear) norm on rank-one matrices.
    
    \section{Boundedness of minimizers}\label{sec:Bounded}

    In this section we consider minimizers of
    $\cE$ with $\cR(w)=\int_\Omega\varrho(Dw)$, $\varrho$ satisfying the assumptions of Theorem~\ref{thm:diff_ex} (whence \textit{\ref{Hdiff}} holds), and     $\cF(w) = \int_\Omega\psi(w-f)$ with $\psi$ convex.
    In order to show that Theorem~\ref{thethm} applies, we need to check that \textit{\ref{Hbv}} and \textit{\ref{Hbdd}} are also satisfied.
    We first assume that $\varrho$ is coercive
    ($\varrho(A)\ge c(|A|-1)$ for some $c>0$), so that \textit{\ref{Hbv}} trivially
    holds. As for~\textit{\ref{Hbdd}}, 
    the situation is trivial if $\psi$ is Lipschitz. Otherwise, as already mentioned in Remark~\ref{remthethm}, we can assume that $f$ is bounded and ensure that the domain of $\cR$ is contained in $L^\infty(\Omega)^n$ by imposing a box constraint. 

That being said, let us now consider the case of unconstrained functional $\cR$ given by \eqref{R_rho}. 

\paragraph{Scalar case} The easiest is the scalar case $n=1$. 
\begin{lemma}\label{lem:bound1D}
    Let $n=1$ and assume that $\psi$ is coercive, that is $\lim_{t\to\pm\infty} \psi(t)=\infty$. Let $f\in L^\infty(\Omega)$ and let $u\in BV(\Omega)$ be a minimizer of $\cE$. Then $u\in L^\infty(\Omega)$.
\end{lemma}
\begin{proof}
By assumption, there is $T>0$ such that $\psi$ is decreasing on $]-\infty,T[$ and
increasing on $]T,\infty[$. Let $u^M:= M \wedge (u \vee -M)$
(the function $u\vee v$ is $x\mapsto \max\{u(x),v(x)\}$, and similarly
$u\wedge v$ is the minimum of $u$ and $v$), and let $M>T+\|f\|_{L^\infty(\Omega)^n}$.
Suppose that $|u|>M$ on a set $E \subset \Omega$ of positive measure. For a.\,e.\ $x \in E$, if $u(x)>M$, then $u(x)-f(x) > u^M(x)-f(x)> T$ 
and if $u(x)<-M$, then $u(x)-f(x) < u^M(x)-f(x) < -T$, whence
$\psi(u(x)-f(x))> \psi(u^M(x)-f(x))$. It follows that $\mathcal{F}(u^M-f)< \mathcal{F}(u-f)$.

On the other hand, it is well known that $\cR(u^M)<\cR(u)$, unless $u^M=u$:
{this can be deduced from the Chain Rule~\cite{ChainRule} which shows
that:
\[
\varrho(Du^M) = \varrho(D^a u)\chi_{\{|u|\le M\}}
+ \varrho(D^c u)\chi_{\{|\widetilde u|\le M\}}
+ ((u^M)^+-(u^M)^-)\varrho^\infty(\nu_u)\cH^{m-1}\mres J_u,  
\]
and is strictly below $\varrho(Du)$ if $u^M\neq u$. (Here, $\widetilde u$ denotes the precise representative, see Section~\ref{sec:BV}.)}
It follows that $\cE(u^M) < \cE(u)$, a contradiction.
\end{proof}

 \paragraph{Vectorial case}   The vectorial case is more complicated. A criterion for having a maximum principle
    in vectorial variational problems is identified in~\cite{LeonettiSiepe2005a,LeonettiSiepe2005b}. Our criterion for
    the regularizer is derived from these references, and ensures that
    when $u$ is projected on some half-space in some (at least $n$) directions, then
    $\cR$ will decrease. For the data term,
    we need also that $\cF(u-f)$ decreases along certain projections, which is
    ensured for instance if $\psi$ it is uniformly coercive in such directions, in the sense which we propose below. We assume that there exist $(e_1,\dots, e_n)$ independent unit vectors of $\bR^n$ such that for $i=1,\dots,n$: 
\begin{itemize}
\item[(i)] $\varrho((I-e_i\otimes e_i)A)\le \varrho(A)$ for all $A\in \bR^{n\times m}$ (with strict inequality if $A^T e_i \neq 0$);
\item[(ii)] Letting, for $u'\in e_i^\perp$, 
\[t_i(u'):= \sup \left\{|t_*| \colon t_* \in \arg\min_{t\in \bR}\psi(u'+te_i)\right\},\] one has $T_i:=\sup_{u'\in e_i^\perp} t_i(u') <\infty$.
\end{itemize}
When (ii) holds, we observe that $t\mapsto \psi(u'+te_i)$ is increasing for $t>T_i$ and decreasing for $t<-T_i$ by convexity of $\psi$.

Thus, we can reproduce the proof of Lemma~\ref{lem:bound1D} in
each direction $e_i$, using 
\[u(x)- (e_i\cdot u(x)- M)_+ e_i - (e_i\cdot u(x)+ M)_- e_i\] 
in place of $M \wedge (u \vee -M)$. We obtain:
\begin{lemma}
    Assume that (i) and (ii) above hold. Let $f\in L^\infty(\Omega)^n$ and let $u\in BV(\Omega)^n$ be a minimizer of $\cE$. Then $u\in L^\infty(\Omega)^n$.
\end{lemma}

\paragraph{Examples} For all the examples in Section~\ref{sec:examples},
(i) is true in all directions of the canonical basis (and in all directions
for the Frobenius norm, or the Schatten norms or other symmetric and non-decreasing function of the singular values, see~\cite[Prop.~6.4]{Bhatia}). Hence
Theorem~\ref{thethm} holds for minimizers of $\cE$, for many data terms
such as data terms of the form $\psi(w) = \widetilde{\psi}(|w|)$ with $\widetilde{\psi}$ a non-negative and coercive convex function, since in such cases (ii) is easy to check.

\paragraph{A remark on the non coercive case}    
We give an (elementary) example here of a situation where the energy is not even
coercive in $BV(\Omega)^n$ and yet, Theorem~\ref{thethm} still applies.
We consider $\omega\subset\bR^{m-1}$ and $\Omega = \omega\times \mathbb{S}^1 \subset \bR^m/(\{0_{m-1}\}\times \mathbb{Z})$ a $m$-dimensional open set which is periodic in the last variable $x_m$.
For simplicity we set $n=1$.
We let $\varrho(p) = \sqrt{\sum_{i=1}^{m-1} p_i^2}$, for $p\in\bR^m$, and to simplify $\psi(t)=t^2/2$.
Then, let $f\in BV(\Omega)\times L^\infty(\Omega)$ and $u$ be the unique minimizer
of
\[
\cE(w) = \varrho(Dw)(\Omega) + \frac{1}{2}\int_\Omega (w-f)^2.
\]
Observe that a minimizing sequence is bounded in $L^2(\Omega)$ and will
converge (weakly) to some limit. The regularizer being convex and lower semicontinuous, a minimizer exists (and is unique by strict convexity).
A priori, $u$ is not necessarily in $BV$, as only the first $(m-1)$ components of $Du$ are
bounded measures. However, just as in the case of the $TV$ regularizer, the subdifferential $\partial \cR$ defines an accretive operator on $L^p(\Omega)$ for $1\le p\le \infty$ (see e.\,g.\ \cite[Appendix A.4]{ACM} for the definition). Thus, the minimization problem for $\cE$ (which coincides with the resolvent problem for $\partial \cR$)
is non-expansive on any $L^p(\Omega)$, $1\le p\le \infty$. This can be shown by a direct calculation using a smooth, uniformly convex approximation of $\cR$. 
Hence, one has for any $t\in\bR$, denoting $e_m=(0,\dots,0,1)$:
\[
\int_\Omega |u(x+te_m)-u(x)|\dd x \le \int_\Omega |f(x+te_m)-f(x)|\dd x \le |t|\int_\Omega |Df|
\]
and we deduce that $e_m\cdot Du$ is also a bounded measure. Alternatively, this can be showed by testing differentiated Euler-Lagrange equation satisfied by $u$ with $\frac{Du}{|Du|}$, reasoning as in \cite{LasicaRybka}. Again, this needs to be made rigorous by regularization.
As clearly one also has $\|u\|_\infty \le \|f\|_\infty<\infty$, it follows that
$u$ satisfies the assumptions of Theorem~\ref{thethm} and we deduce
$u^+-u^-\le f^+-f^-$ $\cH^{m-1}$-a.\,e.~in $\Omega$.

\section{Higher order regularizers}\label{sec:TGV}
\subsection{Inf-convolution based higher order regularizers}
    Here we consider regularizers of form 
    \begin{equation} \label{inf_conv} 
    \cR(w) = \inf_{z \in \cD_2}\widetilde{\cR}(w,z). 
    \end{equation}
    where $\cD_2$ is, in general, a set and $\widetilde{\cR}\colon L^1_\loc(\Omega)^n\times \cD_2 \to [0, \infty]$ is convex. This includes the following several variants of $TV$ of inf-convolution type, introduced in literature to remedy the phenomenon of staircasing observed in solutions to \eqref{eq:ROF}. For simplicity, we recall their form in the case $n=1$.  
    \begin{itemize}
        \item \emph{Total generalized variation (of second order)} 
        \[TGV(w) = \min_{z \in BD(\Omega)} |Dw - z^T \cL^m|(\Omega) + |Ez|(\Omega),\]
        where $Ez = \frac{1}{2}(Dz + Dz^T)$ is the symmetrized gradient and 
        \[BD(\Omega) = \{z \in L^1(\Omega)^m \colon Ez \in M(\Omega)^{m \times m} \}\]
        is the space of functions bounded deformation~\cite{TemamPlasticite}. 
        \item \emph{Non-symmetrized} variant of $TGV$, 
        \[nsTGV(w) = \min_{z \in BV(\Omega)^m} |Dw - z^T \cL^m|(\Omega) + |Dz|(\Omega). \]
        \item \emph{Infimal convolution total variation (of second order)}, 
        \[ICTV(w) = \min_{z \in BV^2(\Omega)} |Dw - Dz|(\Omega) + |DDz|(\Omega),\]
        where 
        \[BV^2(\Omega) = \{z \in BV(\Omega) \colon Dz \in BV(\Omega)^m\}.\] 
    \end{itemize}   
    We will produce a version of Theorem \ref{thethm} that applies to smooth
    variants of all these examples. In our current setting, varying the whole functional $\cR$ in the direction of variable $w$ is not a natural approach. Instead, we will use the formal equivalence of the minimization problem \eqref{eq:mainpb} for $\cE$ with the problem of finding $u, v$ such that 
    \begin{equation}\label{inf_aux}
        \widetilde{\cE}(u,v) = \inf\{\widetilde{\cE}(w,z) \colon w \in L^1_\loc(\Omega)^n, \ z \in \cD_2 \}, \quad \text{where } \widetilde{\cE}(w,z) := \widetilde{\cR}(w, z) + \cF(w-f)
    \end{equation}
    and consider suitable variations that move both $w$ and $z$. Another issue, appearing for example in the case of $TGV$, is that $\cD_2$ might not be closed under (directed) inner variations. On the bright side, we do not to need to assume any particular form of variation in direction of $v$, since we are not interested in obtaining bounds on the part of the minimizer corresponding to the auxiliary variable. 
    \begin{thm}\label{thm:inf_conv}
        Suppose that $\cR$ is of form \eqref{inf_conv}, $f \in BV_\loc(\Omega)^n$ and $u$ is a minimizer of $\cE$ with $\cE(u) < \infty$. In addition to \ref{Hbv} and \ref{Hbdd}, we assume that
    \begin{enumerate}[label=(H\arabic*')] \setcounter{enumi}{\value{Hdiffc}} \addtocounter{enumi}{-1}
    \item \label{Hdiff'}  there exists $v \in \cD_2$ such that $(u,v)$ is a solution to \eqref{inf_aux} and for any directional inner variation $\varphi$ there exists a map 
    \[\tau \mapsto v_{\varphi, \tau} \in L^1_\loc(\Omega)^n \quad \text{with } v_{\varphi,0} = v \]
    defined on a neighborhood of $0$ such that $\tau \mapsto \widetilde{\cR}(u^\varphi_\tau, v_{\varphi,\tau})$ is differentiable at $\tau = 0$.  
    \end{enumerate}
    If $\psi$ is $C^1$ and strictly convex, then $\cH^{m-1}(J_u\setminus J_f)=0$. If $\psi \in C^2$, then 
	\begin{equation} \label{main_ineq_inf} 
		 (u^+ - u^-) \cdot A \, (u^+ - u^-) \leq (f^+ - f^-)\cdot A \, (u^+ - u^-)\qquad \cH^{m-1}\text{-a.\,e.~on } J_u,
	\end{equation} 
where 
\[A = \int_0^1 D^2 \psi (u^- - f^- + s(u^+ - f^+ - u^- + f^+)) \dd s.\] 
    \end{thm} 
The proof of this result is identical to the proof of Theorem~\ref{thethm}, once
the following lemma has been established. (Lemma \ref{lem:fidelity} can be applied directly, since it only concerns the fidelity term, and does \emph{not} assume that $u$ is a minimizer of $\cE$.) 
    \begin{lemma} \label{lem:inner_diff_inf}
Let $u$ be the minimizer of $\cE$ and assume that condition \ref{Hdiff'} of Theorem \ref{thm:inf_conv} is satisfied. For $\vartheta \in [0,1]$, we denote 
\[u^\varphi_{\vartheta, \tau} = \vartheta u_\tau^\varphi + (1 -\vartheta) u.\]
Then, for $\vartheta \in [0,1]$,
\[\liminf_{\tau \to 0^+} \tfrac{1}{\tau}(\cF(u^\varphi_{\vartheta,\tau} - f) - \cF(u - f)) + \liminf_{\tau \to 0^+} \tfrac{1}{\tau}(\cF(u^\varphi_{\vartheta,-\tau} - f) - \cF(u - f)) \geq 0.\]
\end{lemma} 
\begin{proof} 
The proof is the same as the proof of Lemma \ref{lem:inner_diff}, mutatis mutandis. For $\vartheta \in [0,1]$, we denote 
\[v_{\varphi,\vartheta, \tau} = \vartheta v_{\varphi, \tau} + (1 -\vartheta) v.\] 
By minimality of $(u,v)$,
	\begin{multline*} 
 0 \leq \liminf_{\tau \to 0^+}\tfrac{1}{\tau}(\widetilde{\cE}(u^\varphi_{\vartheta, \pm \tau}, v_{\varphi,\vartheta, \pm \tau}) - \widetilde{\cE}(u,v) ) \\ \leq \liminf_{\tau \to 0^+}\tfrac{1}{\tau}(\widetilde{\cR}(u^\varphi_{\vartheta, \pm\tau},v_{\varphi, \vartheta,\pm\tau}) - \widetilde{\cR}(u,v) +\cF(u^\varphi_{\vartheta,\pm\tau} - f) - \cF(u - f)).
 \end{multline*} 
 By our assumption, the function $\widetilde{R}_\varphi\colon ]\!-\!\tau_0, \tau_0[\to [0, \infty[$ defined by 
	\[\widetilde{R}_\varphi(\tau) = \widetilde{\cR}(u^\varphi_\tau, v_{\varphi, \tau})\]
	for $\tau_0$ small enough is differentiable at $\tau = 0$. Thus, by convexity of $\widetilde{\cR}$, 
	\[\tfrac{1}{\tau}(\widetilde{\cR}(u^\varphi_{\vartheta, \pm \tau}, v_{\varphi,\vartheta, \pm \tau}) - \widetilde{\cR}(u,v)) \leq \tfrac{\vartheta}{\tau}(\widetilde{\cR}(u^\varphi_{\pm\tau}, v_{\varphi,\pm \tau}) - \widetilde{\cR}(u,v)) \to \pm \vartheta \widetilde{R}_\varphi'(0) \quad \text{as } \tau \to 0.\]
	Therefore, 
	\begin{equation*} 
		0 \leq \pm \vartheta \widetilde{R}_\varphi '(0) + \liminf_{\tau \to 0^+} \tfrac{1}{\tau}(\cF(u^\varphi_{\vartheta,\pm\tau} - f) - \cF(u - f)). 
	\end{equation*} 
	We conclude by summing together the two obtained inequalities. 
\end{proof}

\subsection{Application}
Now we will discuss conditions under which regularizers of form \eqref{inf_conv} satisfy condition \textit{\ref{Hdiff'}} in the case that
\begin{equation} \label{inf_conv2}\widetilde{\cR}(w,z) = \cR_1(w,z) + \cR_2(z), \text{ where } \cR_1(w,z) = \varrho_1(Dw - z^T \cL^m)(\Omega) 
\end{equation} 
and $\cR_2 \colon \cD_2 \to [0, \infty]$ is given by one of the following 
\begin{itemize}
\item $\cR_2(z) = TD_{\varrho_2}(z) = \varrho_2(Ez)(\Omega)$, $\cD_2 = BD(\Omega)$, 
\item $\cR_2(z) = TV_{\varrho_2}(z) = \varrho_2(Dz)(\Omega)$, $\cD_2 = BV(\Omega)^m$, 
\item $\cR_2(z) = TV_{\varrho_2}(z) = \varrho_2(Dz)(\Omega)$, $\cD_2 = \{z \in BV(\Omega)^m\colon z^T = D \widetilde{z},\ \widetilde{z} \in BV(\Omega)\}$.
\end{itemize} 
For simplicity, we only consider the case $n=1$ here. We assume that $\varrho_1$, $\varrho_2$ are convex. If $\varrho_1 = |\cdot|$, $\varrho_2 = |\cdot|$, $\cR$ coincides with $TGV$, $nsTGV$ and $ICTV$ respectively. However, we are unable to show that \textit{\ref{Hdiff'}} holds in those cases. Instead, we need to consider partially regularized versions of those functionals.
As before, we make the assumption that $\varrho_1$ satistfies~\eqref{lin_growth},
while we make the assumption that $\varrho_2$ has growth $p\ge 1$: there exist
$C_1$, $C_2$, with:
\begin{equation}\label{eq:grho}
C_1 (|M|^p-1) \le \varrho_2(M)\le C_2(|M|^p+1)
\end{equation}
for any $M\in \bR^{m\times m}$.

\begin{thm} \label{thm:diff_ex2} 
Let $\widetilde{\cR}$ be given by \eqref{inf_conv2}. Assume that \ref{Diffac} and \ref{Diffs} from Theorem \ref{thm:diff_ex} hold with $\varrho = \varrho_1$ and that $\varrho_2$, $\varrho_2^\infty$ are differentiable. For $z \in \cD_2$, $\varphi \in C^\infty_c(\Omega)^m$ and $\tau$ in a neighborhood of $0$ we set 
\[z_{\varphi, \tau}(x) = (I + \tau D\varphi(x))^T z(x + \varphi(x)). \]
Then for any $w \in BV(\Omega)$, $z \in \cD_2$ the map $\tau \mapsto \widetilde{\cR}(w^\varphi_\tau, z_{\varphi,\tau})$ is differentiable at $\tau = 0$. 
\end{thm}
\begin{remark} In case $p=1$, one may assume that $\varrho_2^\infty$ is differentiable
at rank-one matrices only (``rank-one symmetric'' matrices [which may have rank 2] for the case of $TD_{\varrho_2}$).
\end{remark}
As in the case of Theorem \ref{thm:diff_ex}, the proof of Theorem \ref{thm:diff_ex2} relies on change of variables formulae similar to \eqref{rep_formula}. 
\begin{lemma}\label{lem:change_var2} 
    Let $F \colon \Omega \to \Omega$ be a diffeomorphism ($C^2$ up to the boundary) and let $w \in BV(\Omega)$, $z \in \cD_2$. Then
    \begin{equation}\label{rep_formula21} \cR_1(w \circ F, DF^T \, z \circ F ) = \int_\Omega |\det D (F^{-1})|\, \dd \varrho_1((Dw - z^T\cL^m)\, DF\circ F^{-1}),  
\end{equation}
\begin{equation} \label{rep_formula22}
TV_{\varrho_2}(DF^T\,z \circ F) =  \int_\Omega |\det D (F^{-1})|\, \dd \varrho_2\left(DF^T\circ F^{-1}\, Dz\, DF\circ F^{-1} + D^2F^T\circ F^{-1}\, z \cL^m\right), 
\end{equation} 
\begin{equation} \label{rep_formula23}
TD_{\varrho_2}(DF^T z \circ F) = \int_\Omega |\det D (F^{-1})|\, \dd \varrho_2\left(DF^T\circ F^{-1}\, Ez\, DF\circ F^{-1} 
+ D^2F^T\circ F^{-1} \,z \cL^m\right). 
\end{equation} 
\end{lemma} 

\begin{proof} 
The proof follows along the lines of Lemma \ref{lem:change_var}. In the case of \eqref{rep_formula21}, we take as before a sequence $(w_k) \subset C^\infty(\Omega)$ that converges weakly-$*$ in $BV(\Omega)$ to $w$ and show that $w_k \circ F \weaklystar w \circ F$ in $BV(\Omega)$. We also take $(z_k) \subset C^\infty(\Omega)^m$ such that $z_k  \to z$ in $L^1(\Omega)^m$; then $z_k\circ F  \to z\circ F$ in $L^1(\Omega)^m$ as well. For any $h \in C_c(\Omega)^m $, we calculate 
    \begin{multline*} 
    \int_\Omega h \circ F \cdot (D (w_k \circ F) - (DF^T z_k\circ F)^T)  = \int_\Omega h \circ F \cdot (Dw_k \circ F - z_k^T \circ F) DF \\ = \int_\Omega h \cdot (Dw_k - z_k^T) DF \circ F^{-1} |\det D(F^{-1})|.     
    \end{multline*} 
    Passing to the limit $k\to \infty$, 
    \begin{equation*} 
    \int_\Omega h \circ F \cdot \dd (D (w \circ F) + (DF^T z\circ F)^T \cL^m) = \int_\Omega |\det D(F^{-1})| h \cdot \dd (Dw + z^T \cL^m) \, DF \circ F^{-1} . 
    \end{equation*} 
    Using \eqref{dual_rep} as before we deduce \eqref{rep_formula21}.     Demonstrations of \eqref{rep_formula22} and \eqref{rep_formula23} again follow the same footsteps. 
    Let us give some details for the case of \eqref{rep_formula23}. We refer to the proof of~\cite[Thm.~3.2]{TemamPlasticite} which closely follows~\cite[1.17]{Giusti}, \cite{AnzellottiGiaquinta} to assert the (weak-$*$) density of smooth functions in $BD(\Omega)$. Then, we will check that the left multiplication by $DF^T$ in the change of variable $DF^T z\circ F$ is the crucial point which ensures that the symmetrized gradient of the transported function can be estimated in terms of $Ez$ only, and $z$ (which is in $L^{m/(m-1)}$ thanks
    to Korn--Poincar\'e's inequality, see~\cite[Sec.~1.2]{TemamPlasticite}), and does not depend on the skew-symmetric part of $Dz$ which is not controlled. 
    We assume without loss of generality that for any $m\times m$ matrix $\xi$, $\varrho_2(\xi)=\varrho_2((\xi+\xi^T)/2)$ depends only on the symmetric part of $\xi$. In that case, \eqref{dual_rep} may be written:
    \begin{equation*}
    \int_\Omega\varphi \dd\rho_2(Ez) = \sup_{h} \int_\Omega \varphi h \cdot \dd Dz - \int_\Omega \varphi \varrho^*(h) \dd \cL^m ,
    \end{equation*} 
    with the $\sup$ taken on all smooth, \textit{symmetric-valued} $h$ with compact
    support. Then, assuming first that $z$ is smooth, we write (summing implicitly on repeated indices):
    \begin{align*}
     \int_\Omega h\circ F \cdot D(DF^T & z\circ F)\dd x
    = \int_\Omega h_{i,j}(F) \partial_i (\partial_j F_k z_k(F)) \dd x\\
    = & \int_\Omega h_{i,j}(F) \left(\partial_{i,j}^2 F_k z_k(F) + 
    \partial_j F_k \partial_i F_l \partial_l z_k(F)
    \right)\dd x \\
    = &
    \int_\Omega |\det D(F^{-1})|h_{i,j} \left(\partial_{i,j}^2 F_k\circ F^{-1} z_k + 
    \partial_j F_k\circ F^{-1} \partial_i F_l \circ F^{-1} (Ez)_{k,l}
    \right)\dd y 
    \end{align*}
    where in the last line we have
    used that $h_{i,j}(F) \partial_i F_l \partial_j F_k$ is symmetric in $(k,l)$.
    As before, we deduce (by approximation with smooth functions) that
    this still holds for any $z\in BD(\Omega)$. Substracting~\eqref{change_star} from
    both sides and taking the supremum over $h$, 
    we deduce~\eqref{rep_formula23}.
    \end{proof}

\begin{proof}[Proof of Theorem \ref{thm:diff_ex2}]
The proof follows along the lines of Theorem \ref{thm:diff_ex}. In the case of $ICTV$-type regularizer we need to note that if $z = D\widetilde{z}^T$ with $\widetilde{z} \in BV(\Omega)$, then $DF^T z \circ F = D(\widetilde{z} \circ F)^T$, in particular $DF^T z \circ F \in \cD_2$. We detail the proof in the $TD_{\rho_2}$ case and leave the other cases to the reader.
We consider diffeomorphisms of the form $F_\tau(x) = 
 x+\varphi(x)\nu$ for $\tau\in\bR$ (small), $\nu$ a unit vector and $\varphi$ a smooth function with compact support.
The term $\cR_1$ will be differentiable as before, so we consider $\cR_2$, which decomposes as:
\begin{multline}\label{eq:difficultterm}
\int_\Omega \det D(F_\tau^{-1}(x))\varrho_2( D^2F^T_\tau(F_\tau^{-1}(x))z(x) 
+ DF^T_\tau(F_\tau^{-1}(x)) e(z)(x) DF_\tau(F_\tau^{-1}(x)) \dd x
\\ +
\int_\Omega \det D(F_\tau^{-1}(x))\varrho_2^\infty(DF_\tau(F_\tau^{-1})^T M_z DF_\tau(F_\tau^{-1})) \dd |E^s z|
\end{multline}
where $M_z$ is the matrix in the polar decomposition of $Ez$ with respect to $|Ez|$.

In case $p>1$, the second integral is not there since $\varrho_2^\infty\equiv \infty$
and $E(z)$ is absolutely continuous.
On the other hand if $p=1$, that second integral is differentiable at $\tau=0$
as soon as $\varrho_2^\infty$ is differentiable at non-zero rank-one symmetric
matrices, since $M_z$ has such structure $E^s z$-a.\,e.~thanks 
to~\cite[Thm. 2.3]{DePhilippisRindler}.

Differentiating the first integral is more subtle.
Indeed, now, if $e(z)(x)=0$, the term in the absolutely continuous integral
does not vanish and is given by $\det D(F_\tau^{-1})(x) \varrho_2(\tau z(x)\cdot\nu D^2\varphi(x))$ which is not differentiable if $\varrho_2$ is not differentiable
at $0$, for instance in the one-homogeneous case of the standard ``$TGV$''.
Assuming that $\varrho_2$ is $C^1$, then, one can write:
\begin{align*}
\int_\Omega &\det D(F_\tau^{-1})\varrho_2( D^2F^T_\tau(F_\tau^{-1})z + DF_\tau(F_\tau^{-1})^T e(z) DF_\tau(F_\tau^{-1})) \dd x
\\ & = \int_\Omega \varrho_2(e(z))\dd x 
\\ & + \int_\Omega (\det D(F_\tau^{-1})-1)\varrho_2(D^2F^T_\tau(F_\tau^{-1})z + DF_\tau(F_\tau^{-1})^T e(z) DF_\tau(F_\tau^{-1})) \dd x
\\ & + \int_0^\tau 
\int_\Omega 
D\varrho_2\Big( e(z)+ s (z\cdot\nu D^2 \varphi + 2(e(z) \nu)\odot \nabla\varphi)
+ s^2 (e(z)\nu)\cdot\nu \nabla\varphi\otimes\nabla\varphi\Big)
\\ &\hspace{3cm} \cdot \Big( z\cdot\nu D^2 \varphi + 2(e(z) \nu)\odot \nabla\varphi
+ 2s (e(z)\nu)\cdot\nu \nabla\varphi\otimes\nabla\varphi\Big) \dd x\dd s,
\end{align*}
where the notation $a\odot b$ stands for the symmetric tensor product
$(a\otimes b + b\otimes a)/2$.
For any two matrices $A,B$,
\[
\varrho_2(A\pm B) -\varrho_2(A) \ge \pm D\varrho_2(A)\cdot B
\]
and one deduces from the growth assumption~\eqref{eq:grho}
that there is $C>0$ such that
\[
|D\varrho_2(A)\cdot B| \le C(|A|^p+|B|^p + 1).
\]
This allows to bound the integrand in the last formula by $C' (1+|e(z)|^p+|z|^p)\in L^1(\Omega)$ (again, thanks to Korn or Poincar\'e--Korn's inequality) for some constant $C'>0$, and apply Lebesgue's dominated convergence to deduce that~\eqref{eq:difficultterm} is differentiable at $\tau=0$.
\end{proof}
\begin{remark}
    If $\varrho_2$ is $1$-homogeneous, we do now know whether the result holds.
    It is however likely that the condition \textit{\ref{Hdiff'}} is not general enough
    to lead to a conclusion, and that one might need a more complicated decomposition
    of the functions, as suggested in~\cite{Valkonen2017}. On the other hand
    the result in the cases $p>1$ is already proved in that reference.
\end{remark}

\section{Data of unbounded variation}\label{sec:nonBV}

In this section, we discuss the case where $f\not\in BV(\Omega)^m$ and we only address the simplest case where the data term is strongly convex and with Lipschitz gradient, that is, verifies~\eqref{sconv}. We introduce a weaker description of a ``jump set'' (which for $BV$ functions coincides with the standard jump set up
to a negligible set), for which we are still able to deduce jump inclusion. 
For $f \in L^2_\loc(\Omega)^n$, $x_0 \in \Omega$, $\nu_0 \in \bS^{m-1}$ we define 
\[j_{f, \nu_0}(x_0)^2 := \limsup_{\tau \to 0^+} \fint_{Q_\tau^-(x_0, \nu_0)} |f(x + \tau \nu_0) - f(x)|^2 \,\dd x, \qquad j_f(x_0):= \sup_{\nu_0 \in \bS^{m-1}}j_{f, \nu_0}(x_0) . \]
(See~\eqref{eq:notation} for the notation $Q^\pm_r(x,\nu)$.)
We denote by $\widetilde{J}_f$ the set of $x_0 \in \Omega$ such that $j_f(x_0) > 0$. 

\begin{prop} \label{prop:L1jump}
For $f \in L^2_\loc(\Omega)^n$, $x_0 \in \Omega$, $\nu_0 \in \bS^{m-1}$ we have 
\[j_{f, \nu_0}(x_0)^2 \leq 4 \limsup_{\tau \to 0^+}  \fint_{Q_\tau(x_0, \nu_0)} \bigg|f - \fint_{Q_\tau(x_0, \nu_0)}f\bigg|^2 .\]
In particular, if $x_0$ is a ($2$-)Lebesgue point of $f$, then $j_f(x_0) = 0$.
If $f \in L^\infty_\loc(\Omega)^n$, then 
\[ J_f \subset \widetilde{J}_f \subset S_f.\]
Moreover, for $x_0 \in J_f$,
$j_f(x_0)=j_{f, \nu_f(x_0)}(x_0) = |f^+ - f^-|(x_0)$. 
\end{prop}
See Section~\ref{sec:SJ} or~\cite[Sec.~3]{afp} for the definition
of $J_f$, $S_f$.

\begin{proof}
First of all, we indeed have 
\begin{multline*}j_{f, \nu_0}(x_0)^2 \leq 2\limsup_{\tau \to 0^+} \fint_{Q^-_\tau(x_0, \nu_0)} \bigg|f(x + \tau \nu_0) - \fint_{Q_\tau(x_0, \nu_0)}f\bigg|^2 + \bigg|f(x) - \fint_{Q_\tau(x_0, \nu_0)} f\bigg|^2\dd x \\ = 2\limsup_{\tau \to 0^+} \fint_{Q^+_\tau(x_0, \nu_0)} \bigg|f(x) - \fint_{Q_\tau(x_0, \nu_0)}f\bigg|^2\dd x + \fint_{Q^-_\tau(x_0, \nu_0)}\bigg|f(x) - \fint_{Q_\tau(x_0, \nu_0)} f\bigg|^2\dd x \\ = 4 \limsup_{\tau \to 0^+}  \fint_{Q_\tau(x_0, \nu_0)} \bigg|f - \fint_{Q_\tau(x_0, \nu_0)}f\bigg|^2.
\end{multline*} 
It is known that Lebesgue points calculated with respect to different \emph{regular families} of sets are the same \cite[Chapter 5, Section 12]{DiBenedetto}. In particular in our case, using Jensen's inequality and observing that $B_{\sqrt{2}\tau}(x_0)\supseteq Q_\tau(x_0, \nu_0)$, we have for every $\nu_0 \in \bS^{m-1}$
\begin{multline*}
  \fint_{Q_\tau(x_0, \nu_0)} \bigg|f - \fint_{Q_\tau(x_0, \nu_0)}f\bigg|^2 \leq 2 \fint_{Q_\tau(x_0, \nu_0)} \bigg|f -\fint_{{B}_{\sqrt{2}\tau}(x_0)}f\bigg|^2
  + 2 \bigg|\fint_{Q_\tau(x_0, \nu_0)} f -
   \fint_{{B}_{\sqrt{2}\tau}(x_0)}f\bigg|^2
   \\ \leq 4 \fint_{Q_\tau(x_0, \nu_0)} \bigg|f - \fint_{{B}_{\sqrt{2}\tau}(x_0) 
   }
   f\bigg|^2 \leq 4 \frac{\cL^m({B}_{\sqrt{2}\tau}(x_0))}{\cL^m(Q_\tau(x_0, \nu_0))}
   \fint_{{B}_{\sqrt{2}\tau}(x_0)} \bigg|f - \fint_{{B}_{\sqrt{2}\tau}(x_0)} f\bigg|^2
\end{multline*} 
Thus, if $x_0$ is a ($2$-)Lebesgue point of $f$, then $j_{f, \nu_0}(x_0) = 0$ for every $\nu_0 \in \bS^{m-1}$, whence $j_{f}(x_0) = 0$. 

Now suppose that $f \in L^\infty_\loc(\Omega)^n$. If $x_0 \in \Omega\setminus S_f$, then 
\[\fint_{B_\tau(x_0)} \bigg|f - \fint_{B_\tau(x_0)}f\bigg|^2 \leq 2 \|f\|_{L^\infty(B_\tau(x_0))} \fint_{B_\tau(x_0)} \bigg|f - \fint_{B_\tau(x_0)}f\bigg| \to 0 \quad \text{as } \tau \to 0^+,  \]
so $j_f(x_0) =0$, i.\,e.\ $x_0 \in \Omega\setminus \widetilde{J}_f$. On the other hand, if $x_0 \in J_f$ and $\nu_0$ is the direction of jump of $f$ at $x_0$, then by the triangle inequality in $L^2(Q^-_\tau(x_0, \nu_0))^n$, 
\begin{multline*} 
\sqrt{\fint_{Q^-_\tau(x_0, \nu_0)} |f(x + \tau \nu_0) - f(x)|^2\, \dd x} \\ \geq -\sqrt{\fint_{Q^+_\tau(x_0, \nu_0)} |f - f^+(x_0)|^2} + |f^+(x_0) - f^-(x_0)| - \sqrt{\fint_{Q^-_\tau(x_0, \nu_0)} |f - f^-(x_0)|^2} . 
\end{multline*} 
Since 
\[\fint_{Q^\pm_\tau(x_0, \nu_0)} |f - f^\pm(x_0)|^2 \leq 2 \|f\|_{L^\infty(Q^\pm_\tau(x_0, \nu_0))^n} \fint_{Q^\pm_\tau(x_0, \nu_0)} |f - f^\pm(x_0)| \to 0 \quad \text{as } \tau \to 0^+, \]
we obtain 
\[ j_f(x_0) \geq j_{f, \nu_0}(x_0)\geq |f^+(x_0) - f^-(x_0)|,\]
in particular $x_0 \in \widetilde{J}_f$. It remains to prove the opposite inequality. Let $\nu \in \bS^{m-1}$ and let $q=q(\nu_0, \nu) \ge 1$ be the smallest number such that $Q_\tau(x_0, \nu) \subset Q_{q\tau}(x_0, \nu_0)$. We stress that $q$ does not depend on $\tau$. Assume without loss of generality that $\nu_0 \cdot \nu \geq 0$. Then $Q^\pm_\tau(x_0, \nu)$ can be divided into six parts (see Figure \ref{fig:cylinders}): 
\[A^{0,\pm}_\tau(x_0, \nu) = \{x \in Q^\pm_\tau(x_0, \nu) \colon x \in Q^\pm_{q\tau}(x_0, \nu_0),\ x \mp \tau \nu \in Q^\mp_{q\tau}(x_0, \nu_0)\},\]
\[A^{+,\pm}_\tau(x_0, \nu) = \{x \in Q^\pm_\tau(x_0, \nu) \colon x \in Q^\pm_{q\tau}(x_0, \nu_0),\ x \mp \tau \nu \in Q^\pm_{q\tau}(x_0, \nu_0)\},\]
\[A^{-,\pm}_\tau(x_0, \nu) = \{x \in Q^\pm_\tau(x_0, \nu) \colon x \in Q^\mp_{q\tau}(x_0, \nu_0),\ x \mp \tau \nu \in Q^\mp_{q\tau}(x_0, \nu_0)\}.\]
\begin{figure}[h]
\begin{center}
  \includegraphics[width=0.4\textwidth]{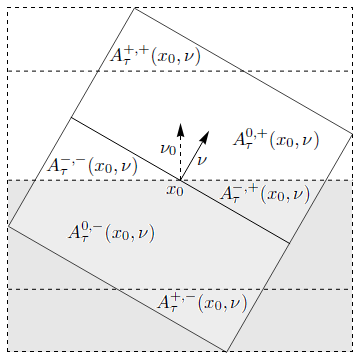}
\caption{Sets $A^{0,\pm}_\tau(x_0, \nu)$, $A^{+,\pm}_\tau(x_0, \nu)$, $A^{-,\pm}_\tau(x_0, \nu)$.}\label{fig:cylinders}
\end{center}
\end{figure}

By definition, $A^{\mp,+}_\tau(x_0, \nu) \cup A^{\pm,-}_\tau(x_0, \nu) \subset Q^\mp_{q\tau}(x_0, \nu_0)$, so 
\begin{multline*} 
\int_{A^{\pm,-}_\tau(x_0, \nu)} |f(x + \tau \nu_0) - f(x)|^2 \, \dd x = \int_{A^{\pm,-}_\tau(x_0, \nu)} |f(x + \tau \nu_0) - f^\mp(x_0) + f^\mp(x_0)- f(x)|^2\, \dd x \\ \leq 2 \int_{A^{\mp,+}_\tau(x_0, \nu)} |f - f^\mp(x_0)|^2 + 2 \int_{A^{\pm,-}_\tau(x_0, \nu)} |f - f^\mp(x_0)|^2 = 2 \int_{A^{\mp,+}_\tau(x_0, \nu) \cup A^{\pm,-}_\tau(x_0, \nu)} |f - f^\mp(x_0)|^2\\ 
\leq 4 \|f\|_{L^\infty(Q^\mp_{q\tau}(x_0, \nu_0))^n} \int_{Q^\mp_{q\tau}(x_0, \nu_0)} |f - f^\mp(x_0)| 
\end{multline*} 
and 
\begin{multline*} 
\frac{1}{\cL^m(Q^-_\tau(x_0,\nu))}\int_{A^{\pm,-}_\tau(x_0, \nu)} |f(x + \tau \nu_0) - f(x)|^2 \, \dd x \\ \leq  4 \|f\|_{L^\infty(Q^\mp_{q\tau}(x_0, \nu_0))^n} q^m\fint_{Q^\mp_{q\tau}(x_0, \nu_0)} |f - f^\mp(x_0)| \to 0 \text{ as } \tau \to 0^+.
\end{multline*} 
Therefore, 
\begin{multline*} 
\limsup_{\tau \to 0^+} \fint_{Q^-_\tau(x_0, \nu_0)} |f(x + \tau \nu_0) - f(x)|^2\, \dd x \\ = \limsup_{\tau \to 0^+} \frac{1}{\cL^m(Q^-_\tau(x_0,\nu))}\int_{A^{0,-}_\tau(x_0, \nu)} |f(x + \tau \nu_0) - f(x)|^2\, \dd x  \\ \leq \limsup_{\tau \to 0^+} \fint_{A^{0,-}_\tau(x_0, \nu)} |f(x + \tau \nu_0) - f(x)|^2\, \dd x. 
\end{multline*} 
Then, by the triangle inequality in $L^2(A^{0,-}_\tau(x_0, \nu))^n$, 
\begin{multline*} 
\sqrt{\fint_{A^{0,-}_\tau(x_0, \nu)} |f(x + \tau \nu_0) - f(x)|^2\, \dd x} \\ \leq \sqrt{\fint_{A^{0,+}_\tau(x_0, \nu)} |f - f^+(x_0)|^2} + |f^+(x_0) - f^-(x_0)| + \sqrt{\fint_{A^{0,-}_\tau(x_0, \nu)} |f - f^-(x_0)|^2} . 
\end{multline*} 
We estimate 
\[\fint_{A^{0,\pm}_\tau(x_0, \nu)} |f - f^\pm(x_0)|^2 \leq 2 \|f\|_{L^\infty(A^{0,\pm}_\tau(x_0, \nu))^n} \frac{\cL^m(Q^\pm_{q\tau}(x_0, \nu_0))}{\cL^m(A^{0,\pm}_\tau(x_0, \nu))} \fint_{Q^\pm_{q\tau}(x_0, \nu_0)} |f - f^\pm(x_0)|. \]
Since the quotient $\cL^m(Q^\pm_{q\tau}(x_0, \nu_0))/\cL^m(A^{0,\pm}_\tau(x_0, \nu))$ is independent of $\tau$, the r.\,h.\,s.\ converges to $0$ as $\tau \to 0^+$, whence 
\[ j_{f, \nu}(x_0)\leq |f^+(x_0) - f^-(x_0)|.\]
As $\nu \in \bS^{m-1}$ is arbitrary, we conclude.  
\end{proof} 

Either one of the inclusions $J_f \subset \widetilde{J}_f$ and $\widetilde{J}_f \subset S_f$ may be strict, even if $f$ is a $BV$ function. For an example of $x \in \widetilde{J}_f \setminus J_f$, take any domain $\Omega \subset \mathbb{R}^2$ containing $0$ and consider $f = \chi_{\Omega \cap [0,1]^2}$. Then, setting $\nu_0 = (0,1)$, 
\[ j_f(0)^2 \geq j_{f,\nu_0}(0)^2 = \limsup_{\tau \to 0^+} \fint_{Q_\tau^-(0, \nu_0)} \chi_{\{x_1 > 0\}}(x) \dd x  = \frac{1}{2} >0. \]
Thus, $0 \in \widetilde{J}_f$. At the same time,  $0 \not \in J_f$ (see~\eqref{jump_def}). 

For an example of $S_f \neq \widetilde{J}_f$, take any bounded domain $\Omega \subset \mathbb{R}^2$ containing $0$ and consider $f \in BV(\Omega)$ given by $f(x) = \left|\log |x|\right|^{1/2}$.  Take any $\nu \in \bS^1$ and $\tau > 0$. Setting $\widehat{Q}^-_\tau(0,\nu) := \left\{ x \in Q_\tau^-(0, \nu)\colon \nu \cdot x > - \tau/2\right\}$, we have $4 \tau^2 > |x + \tau \nu|^2 = |x|^2 + 2 \tau \nu \cdot x + \tau^2 \geq |x|^2$ for $x \in \widehat{Q}^-_\tau(0,\nu)$. 
Thus, using radial symmetry and monotonicity of $f$,  
\begin{multline*} \fint_{Q_{\tau}^-(0, \nu)} \left|f(x + \tau \nu) - f(x)\right|^2\dd x =  \fint_{\widehat{Q}^-_\tau(0,\nu) } \left|f(x + \tau \nu) - f(x)\right|^2\dd x  \\  \leq \fint_{\widehat{Q}^-_\tau(0,\nu)} \left|\left|\log |x|\right|^{1/2} - \left|\log 2\tau\right|^{1/2}\right|^2\dd x .   
\end{multline*} 
We estimate 
\[\left|\log |x|\right|^{1/2} - \left|\log 2\tau\right|^{1/2} = \frac{- \log |x| +  \log 2\tau}{\left|\log |x|\right|^{1/2} + \left|\log 2\tau\right|^{1/2}} \leq - \frac{\log \frac{|x|}{2 \tau}}{2 \left|\log 2\tau\right|^{1/2}} \]
and so, by a change of variables $x = \tau y$, 
\[\fint_{Q_{\tau}^-(0, \nu)} \left|f(x + \tau \nu) - f(x)\right|^2\dd x  \leq \fint_{\widehat{Q}^-_\tau(0,\nu)} \left|\frac{\log \frac{|x|}{2 \tau}}{2 \left|\log 2\tau\right|^{1/2}}\right|^2 \dd x = \frac{1}{4 \left|\log 2\tau\right|} \fint_{\widehat{Q}_1^-(0, \nu)} \left|\log\frac{y}{2}\right|^2 \dd y.\]
Since the integral on the r.\,h.\,s.\ is finite, the whole expression converges to $0$ as $\tau \to 0^+$ independently of $\nu \in \bS^1$. Thus, we have $j_{f}(0) = 0$, whence $0 \in \widetilde{J}_f$. On the other hand, the averages of $f$ over $B_\tau(0)$ converge to $\infty$ as $\tau \to 0^+$, so $0$ cannot be a point of approximate continuity of $f$. 

Modifying this example a bit, one can produce a BV function $f$ satisfying $0 \in S_f \setminus \widetilde{J}_f$ that is also bounded: take $f(x) = \sin \left|\log |x|\right|^{1/2}$. We do not include here the straightforward but relatively lengthy computations justifying this claim. 

It is an interesting question, that is beyond the scope of this paper, how large can the differences $ \widetilde{J}_f \setminus J_f$, $S_f \setminus \widetilde{J}_f$ be for a generic $f \in L^\infty_{loc}(\Omega)^n$. However, as a consequence of Proposition \ref{prop:L1jump} and the Federer--Vol'pert Theorem~\cite[Theorem 3.78]{afp}, if $f \in L^\infty_\loc(\Omega)^n \cap BV_\loc(\Omega)^n$, the three sets $S_f$, $\widetilde{J}_f$ and $J_f$ coincide up to $\cH^{m-1}$-negligible sets. 
\begin{thm} \label{thm:nonbv}
Let $f \in L^\infty(\Omega)^n$, suppose that $\cE$ admits a minimizer $u \in L^\infty(\Omega)^n \cap BV(\Omega)^n$, $\psi$ is $C^2$ and \eqref{sconv} holds on $\{ z\in\bR^n\colon|z|\le \|u\|_{L^\infty(\Omega)^n}+\|f\|_{L^\infty(\Omega)^n}\}$. Assume \ref{Hdiff} or that $\cR$ is of form \eqref{inf_conv} and \ref{Hdiff'} holds. Then $J_u \subset \widetilde{J}_f$ up to a $\cH^{m-1}$-negligible set and 
\[|u^+ - u^-|(x_0) \leq \sqrt{\Lambda/\lambda}\, j_f(x_0) \quad \text{for }\cH^{m-1}\text{-a.\,e.\ } x_0 \in J_u. \]
\end{thm}

\begin{proof} 
By Lemma \ref{lem:inner_diff} (or Lemma \ref{lem:inner_diff_inf} in the inf-convolution setting) we have for any (directional) inner variation $\varphi$ and $\vartheta \in [0,1]$
\begin{equation} \label{fid_liminf} 0 \leq \liminf_{\tau \to 0^+} \tfrac{1}{\tau}(\cF(u^\varphi_{\vartheta,\tau} - f) - \cF(u - f)) +  \tfrac{1}{\tau}(\cF(u^\varphi_{\vartheta,-\tau} - f) - \cF(u - f)) .
\end{equation} 
We take $\Gamma$, $x_0$, $\nu_0 = \nu_u(x_0)$ and $r_0$ as in the beginning of the proof of Theorem \ref{thethm}, except now we cannot assume that the traces $f^\pm$ exist on both sides of $\Gamma$. Instead we assume that $x_0$ is a Lebesgue point of $j_{f}$ with respect to $\cH^{m-1} \mres \Gamma$, in particular
\begin{equation} \label{Leb_point_j}
\fint_{Q_r \cap \Gamma} j_{f}(x)^2 \dd \cH^{m-1}(x) \stackrel{r\to 0}{\longrightarrow} j_{f}(x_0)^2.
\end{equation} 
As in the proof of Lemma~\ref{lem:fidelity}, by an isometric change of coordinates, we assume that $\nu_0 = e_m$, $x_0 = 0$ and denote $x = (x', x_m)$, $Q_r(x_0, \nu_0) = Q_r$, $B^{m-1}_r(x_0,\nu_0) = B^{m-1}_r$, 
\[\Gamma = \left\{\gamma(x')\colon x' \in B^{m-1}_r\right\},\]
$\gamma(x') = (x', \widetilde{\gamma}(x'))$. We recall that we assume $\widetilde{\gamma}$ is $C^1$.
For $s\le r$, we let $L_s = \max_{B^{m-1}_s}|D\widetilde{\gamma}|$ the
Lipschitz constant of $\widetilde{\gamma}$ on $B^{m-1}_s$, which
is such that $\lim_{s\to 0} L_s=0$.

\begin{figure}[h]
\begin{center}
  \includegraphics[width=0.4\textwidth]{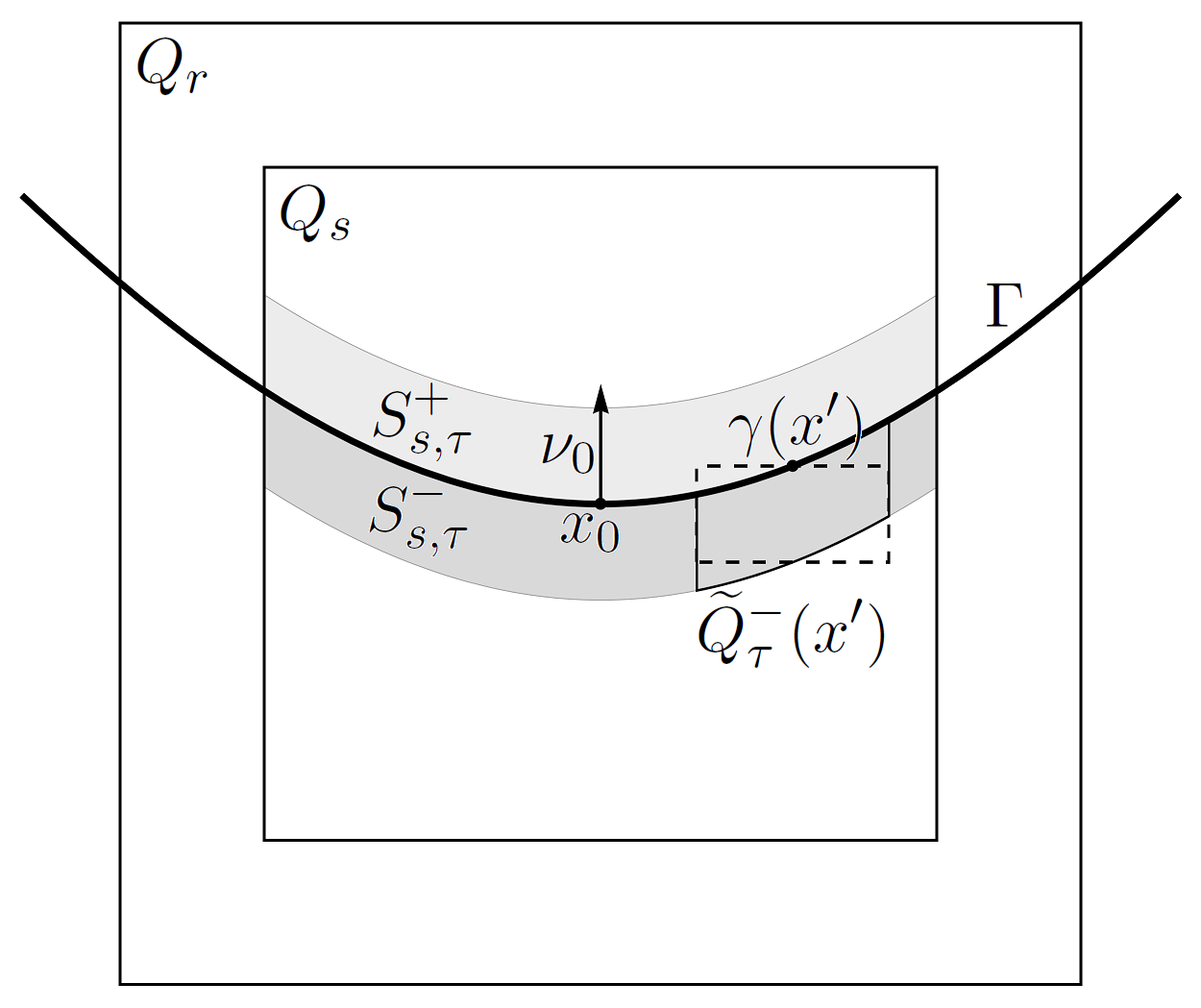}
\caption{Notation in the proof of Theorem \ref{thm:nonbv}. The region bounded by the dashed line is $Q^-_\tau(\gamma(x'), \nu_0)$.}\label{fig:t72}
\end{center}
\end{figure}
Given $0<s<r$, we take $\varphi \in C_c^\infty(\Omega)^n$ such that the support of $\varphi$ is contained in $Q_r$ and $\varphi = \nu_0 \widetilde{\varphi}$, where $0 \leq \widetilde{\varphi} \leq 1$ and $\widetilde{\varphi}=1$ on $\overline{B^{m-1}_s}\times [-r/2,r/2]$. We denote
\[S^-_{s,\tau} = \{(x',x_m) \colon x' \in \overline{B^{m-1}_s}, \ \widetilde{\gamma}(x') - \tau \leq x_m \leq \widetilde{\gamma}(x')\}, \]
\[S^+_{s,\tau} = \{(x',x_m) \colon x' \in \overline{B^{m-1}_s}, \ \widetilde{\gamma}(x') \leq x_m \leq \widetilde{\gamma}(x') + \tau\}, \]
\[u_{\tau}(x) =u(x', x_m + \tau), \quad f_{\tau}(x) =f(x', x_m + \tau), \]
\[ u_{\vartheta,\tau} = \vartheta u_\tau + (1 - \vartheta) u,\quad f_{\vartheta,\tau} = \vartheta f_\tau + (1 - \vartheta) f.\]
We note that $u_\tau = u^\varphi_\tau$ and $u_{\vartheta,\tau} = u^\varphi_{\vartheta,\tau}$ on $S^-_{s,\tau}$ as soon as $\tau\le r/2$. We decompose 
\[\cF(u^\varphi_{\vartheta,\pm\tau} - f) - \cF(u - f) = \int_{S^\mp_{s,\tau}} \!\!\psi(u^\varphi_{\vartheta,\pm \tau} - f)-\psi(u - f) + \int_{Q_r\setminus S^\mp_{s,\tau}} \!\!\psi(u^\varphi_{\vartheta,\pm \tau} - f)-\psi(u - f).\]
Reasoning as in the proof of Lemma \ref{lem:fidelity}
(where we use that $\psi$ is Lipschitz in $\{ z\in\bR^n\colon|z|\le \|u\|_{L^\infty(\Omega)^n}+\|f\|_{L^\infty(\Omega)^n}\}$),
we get 
\begin{equation*}\frac{1}{\tau}\int_{Q_r\setminus S^\mp_{s,\tau}} \!\psi(u^\varphi_{\vartheta,\pm \tau} - f)-\psi(u - f) \le
  \frac{C}{\tau}\int_{Q_r\setminus S^\mp_{s,\tau}} \! |u^\varphi_{\vartheta,\pm \tau} -u|
 \leq
C \vartheta |\nu_0\cdot\D u|(Q_r\setminus (Q_s \cap \Gamma )).
\end{equation*}
On the other hand, using the convexity of $\psi$,
\begin{multline*}
\int_{S^-_{s,\tau}}\psi(u^\varphi_{\vartheta,\tau} - f)-\psi(u - f) + \int_{S^+_{s,\tau}} \psi(u^\varphi_{\vartheta,-\tau} - f)-\psi(u - f) \\ = \int_{S^-_{s,\tau}} \psi(u_{\vartheta,\tau} - f)-\psi(u - f) + \psi(u_{(1-\vartheta),\tau} - f_\tau)-\psi(u_\tau - f_\tau) \\
\leq \int_{S^-_{s,\tau}}\vartheta D\psi(u_{\vartheta,\tau} - f)\cdot (u_\tau - u)  + \vartheta D\psi(u_{(1-\vartheta),\tau} - f_\tau)\cdot (u - u_\tau) \\
= \vartheta\int_{S^-_{s,\tau}} (D\psi(u_{\vartheta,\tau} - f)- D\psi(u_{(1-\vartheta),\tau} - f_\tau))\cdot (u_\tau - u) \\= \vartheta\int_{S^-_{s,\tau}} (u_{\vartheta,\tau} - u_{(1-\vartheta),\tau} - f + f_\tau)\cdot A_{\vartheta, \tau}\cdot (u_\tau - u),  
\end{multline*}
where the symmetric positive definite matrix $A_{\vartheta,\tau}$ is
given by:
\[A_{\vartheta, \tau} = \int_0^1 D^2 \psi (u_{(1-\vartheta),\tau} - f_\tau + s(u_{\vartheta,\tau} - u_{(1-\vartheta),\tau} - f + f_\tau)) \dd s.\]
Using $2\xi\cdot A_{\vartheta,\tau} \cdot\eta\le
\xi\cdot A_{\vartheta,\tau} \cdot\xi
+\eta\cdot A_{\vartheta,\tau}\cdot \eta$ for any $\xi,\eta$ and \eqref{sconv} it follows that
\begin{multline*} 
  (u_{\vartheta,\tau} - u_{(1-\vartheta),\tau} - f + f_\tau)\cdot A_{\vartheta, \tau}\cdot (u_\tau - u) \\ =  - (1 - 2 \vartheta)(u_\tau - u)\cdot A_{\vartheta, \tau}\cdot (u_\tau - u) + (f_\tau - f)\cdot
 {A_{\vartheta, \tau}} 
  \cdot (u_\tau - u) \\ \leq - (\tfrac{1}{2} - 2 \vartheta)(u_\tau - u)\cdot A_{\vartheta, \tau}\cdot (u_\tau - u) + \tfrac{1}{2}(f_\tau - f)\cdot A_{\vartheta, \tau} \cdot (f_\tau - f) \\ \leq - (\tfrac{1}{2} - 2 \vartheta) \lambda |u_\tau - u|^2 + \tfrac{1}{2}\Lambda |f_\tau - f|^2
\end{multline*} 
for $\vartheta \in [0, \tfrac{1}{4}]$. Using \cite[Theorem 3.108]{afp}, recalling \eqref{fid_liminf} and combining the estimates above,  
\begin{multline}\label{int_jump_est}
\vartheta (1 - 4 \vartheta)\lambda\int_{Q_s\cap \Gamma} |u^+ - u^-|^2 \,\dd \gamma_{\#} \cL^{m-1} = \vartheta (1 - 4 \vartheta)\lambda\lim_{\tau \to 0^+} \frac{1}{\tau}\int_{S^-_{s,\tau}} |u_\tau - u|^2 \\ \leq \vartheta \Lambda\liminf_{\tau \to 0^+} \frac{1}{\tau} \int_{S^-_{s,\tau}} |f_\tau - f|^2 + 2 C \vartheta |\nu_0\cdot\D u|(Q_r\setminus (Q_s \cap \Gamma )).
\end{multline}
We recall that
here, $u^\pm$ coincide with the traces of $u$ on both side of $\Gamma$. 
We estimate the pushforward measure $\gamma_{\#}\cL^{m-1}$ (recall \eqref{graph_push}) by
\[\gamma_{\#}\cL^{m-1}\mres(Q_s\cap\Gamma)
  =\frac{\cH^{m-1}\mres(Q_s\cap\Gamma)}{\sqrt{1+|D\widetilde{\gamma} \circ \gamma^{-1}}|^2}\ge \frac{1}{\sqrt{1+L_s^2}}\cH^{m-1}\mres(Q_s\cap\Gamma).\]

Now, for $x' \in B^{m-1}_s$, $\tau < r-s$, let us denote 
\[\widetilde{Q}^-_\tau(x') = \{(y', y_m) \in Q_r \colon |y' - x'| < \tau,\ \widetilde{\gamma}(y') - \tau < y_m < \widetilde{\gamma}(y')\}. \]
We observe that $\cL^m(\widetilde{Q}^-_\tau(x')) = \tau\,\cL^{m-1}(B^{m-1}_\tau)$. Then 
\[\frac{1}{\tau} \int_{S^-_{s,\tau}} |f_\tau - f|^2 \leq \int_{B^{m-1}_s} \left(\fint_{\widetilde{Q}^-_\tau(x')} |f_\tau - f|^2\right) \dd x'.\]
Note that $\fint_{\widetilde{Q}^-_\tau(x')} |f_\tau - f|^2$ is uniformly bounded by $4 \|f\|_{L^\infty(Q_r)^n}^2$. Thus, by Fatou's Lemma
\[\liminf_{\tau \to 0^+} \frac{1}{\tau} \int_{S^-_{s,\tau}} |f_\tau - f|^2 \leq \int_{B^{m-1}_s} \left(\limsup_{\tau \to 0^+}\fint_{\widetilde{Q}^-_\tau(x')} |f_\tau - f|^2\right) \dd x'.\]
Dividing \eqref{int_jump_est} by $\vartheta$ and passing to the limits $\vartheta \to 0^+$, $s \to r^-$,
\begin{equation} \label{int_jump_est2}
\lambda\int_{Q_r\cap \Gamma} |u^+ - u^-|^2 \,\dd \gamma_{\#} \cL^{m-1}  \leq  \Lambda\int_{B^{m-1}_r} \left(\limsup_{\tau \to 0^+}\fint_{\widetilde{Q}^-_\tau(x')} |f_\tau - f|^2\right) \dd x'+ 2 C |\nu_0\cdot\D u|(Q_r\setminus  \Gamma ).
\end{equation}
We estimate 
\[\cL^m\left(\widetilde{Q}^-_\tau(x') \setminus Q^-_\tau(
\gamma(x'), \nu_0)\right)\leq \cL^{m-1}(B^{m-1}_\tau)\, \tau  \max_{B^{m-1}_r} |D \widetilde{\gamma}|
  = L_r \cL^{m}(\widetilde{Q}^-_\tau(x'))\, \tau . \]
Therefore, 
\[\fint_{\widetilde{Q}^-_\tau(x')} |f_\tau - f|^2 \leq \fint_{Q^-_\tau(\gamma(x'), \nu_0)} |f_\tau - f|^2 + 4 L_r\|f\|_{L^\infty(Q_r)^n}^2\]
and
\[
  \limsup_{\tau\to 0^+}  \fint_{\widetilde{Q}^-_\tau(x')} |f_\tau - f|^2 \leq
  j_f(\gamma(x'))^2  + 4 L_r\|f\|_{L^\infty(Q_r)^n}^2.
\]
Hence, we can estimate
\[ \int_{B^{m-1}_r}\!\! \left(\limsup_{\tau \to 0^+}  \fint_{\widetilde{Q}^-_\tau(x')} |f_\tau - f|^2\right) \dd x' 
  \le \int_{B^{m-1}_r}    j_f(\gamma(x'))^2 \dd x' + 4
  \cL^{m-1}(B^{m-1}_r) L_r\|f\|_{L^\infty(Q_r)^n}^2, 
\]
\begin{equation*} 
\int_{B^{m-1}_r}    j_f(\gamma(x'))^2 \dd x' 
= \int_{Q_r\cap \Gamma} \frac{j_f(x)^2}{\sqrt{1+|D\widetilde{\gamma}(x')|^2}}\dd\cH^{m-1}(x) \le \int_{Q_r\cap \Gamma}j_f^2\,\dd\cH^{m-1}.
\end{equation*} 
Recalling \eqref{int_jump_est2}, we deduce 
\begin{multline} \label{int_jump_est3}
  \lambda\int_{Q_r\cap \Gamma} |u^+ - u^-|^2 \,\dd \gamma_{\#} \cL^{m-1}  \leq  \Lambda\int_{Q_r \cap \Gamma} j_{f}^2\, \dd \cH^{m-1}\\ 
+ 4  \cL^{m-1}(B^{m-1}_r) L_r\|f\|_{L^\infty(Q_r)^n}^2
  + 2 C |\nu_0\cdot\D u|(Q_r\setminus  \Gamma ).
\end{multline}
Finally, we divide both sides of \eqref{int_jump_est3} by $\cL^{m-1}(B^{m-1}_r) = \gamma_{\#} \cL^{m-1}(Q_r \cap \Gamma)$ and keeping in mind that $\cH^{m-1}(Q_r \cap \Gamma)/\cL^{m-1}(B^{m-1}_r)\to 1$
and $L_r\to 0$ as $r \to 0^+$, we pass to the limit obtaining the asserted inequality owing to \eqref{Leb_point_j}. 
\end{proof} 

\section{Experiments}\label{sec:experiment}
\begin{figure}[h]
\begin{center}
  \includegraphics[width=.24\textwidth]{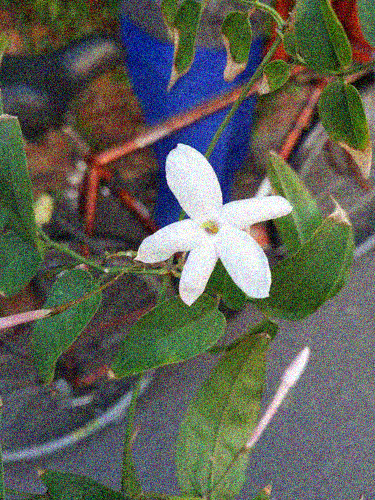}
  \includegraphics[width=.24\textwidth]{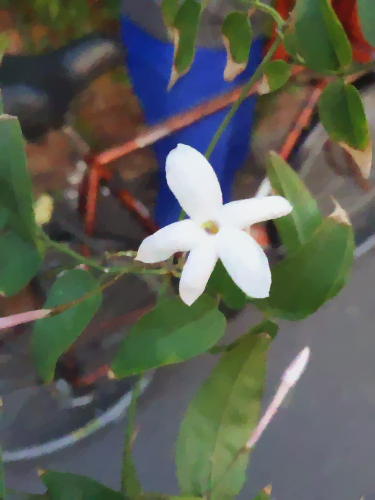}
  \includegraphics[width=.24\textwidth]{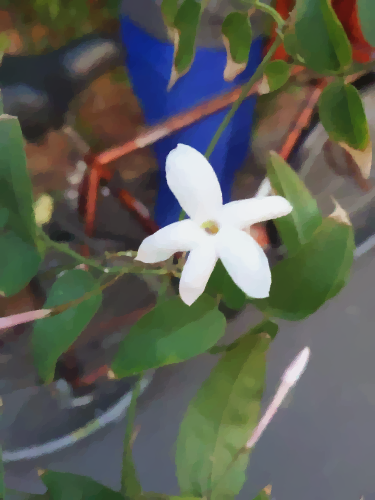}
  \includegraphics[width=.24\textwidth]{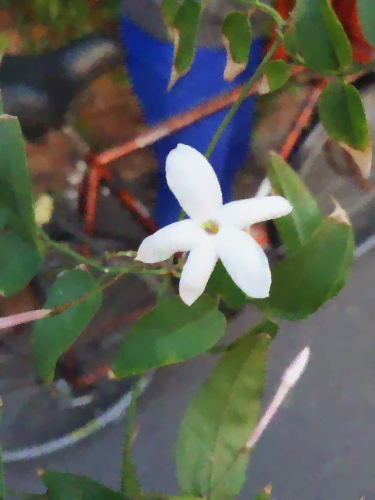}
\caption{A noisy image and the denoised versions with, respectively, the Frobenius ROF, Nuclear ROF, Spectral ROF problems.}\label{fig:flower}
\end{center}
\end{figure}
\begin{figure}[h]
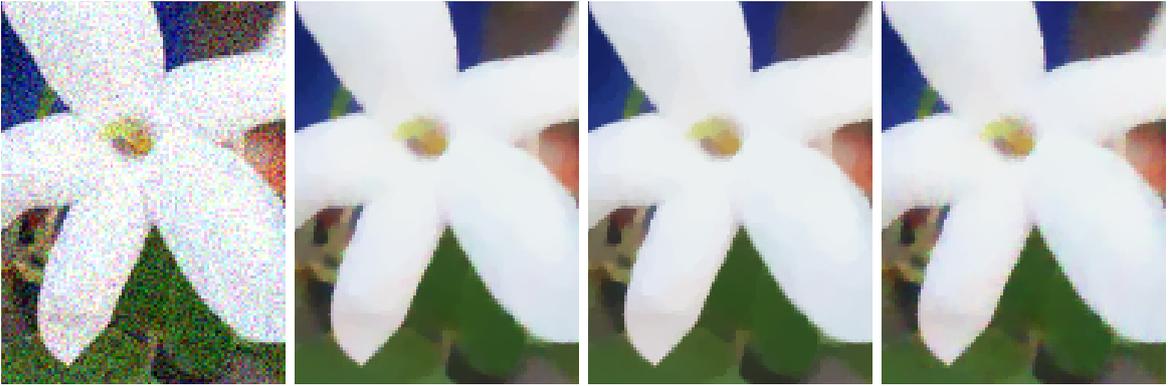

\begin{center} 
  \includegraphics[trim=120 135 80 130,clip,width=.24\textwidth]{fleur_noise}
  \includegraphics[trim=120 135 80 130,clip,width=.24\textwidth]{fln_tv_0_15}
  \includegraphics[trim=120 135 80 130,clip,width=.24\textwidth]{fln_0_15}
  \includegraphics[trim=120 135 80 130,clip,width=.24\textwidth]{fln_sp_0_15}
\caption{Detail of Figure~\ref{fig:flower}}\label{fig:closeup}
\end{center}
\end{figure}
We solved here the ``ROF'' problem~\eqref{eq:ROF} for a data term given by a noisy
color image, and the Frobenius, Nuclear and Spectral total variations. Figure~\ref{fig:flower} shows the results, which look almost identical. The close-up in Figure~\ref{fig:closeup} seems to show that the edges are better recovered with the Nuclear total variation, and quite jagged in the case of the Spectral total variation, for which we do not have a proof of jump inclusion (and since there is no control
on the lower eigenmode of $Du$, it leaves strong oscillations along the edges,
which make this regularizer probably less interesting for such tasks). 
In these experiments,
the noise was Gaussian with standard deviation $.1$ (for RGB values in $[0,1]^3$),
we show in Figures~\ref{fig:flower_morenoise}--\ref{fig:closeup_morenoise} the same experiments with stronger noise (of deviation $.3$) and stronger regularization.  
Of course, this is a discrete experiment at a fixed scale and therefore a relatively poor illustration of our main results. Observe that in these experiments, one cannot expect that the original datum (left image) represents a function $f\in BV(\Omega,[0,1]^3)$. However, being obtained by adding a small amplitude noise to a bounded variation function, we may expect that the set $\widetilde{J}_f$ of Section~\ref{sec:nonBV} corresponds to the set of (large enough) edges in the original image.
\begin{figure}[h]
	\begin{center}
		\includegraphics[width=.24\textwidth]{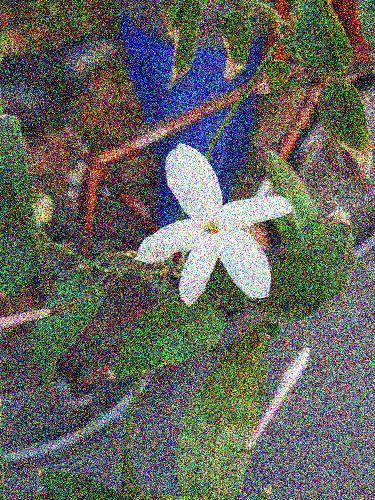}
		\includegraphics[width=.24\textwidth]{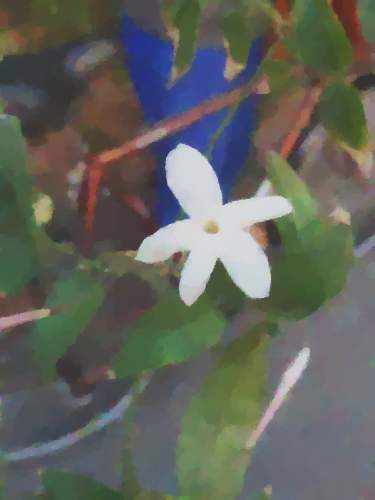}
		\includegraphics[width=.24\textwidth]{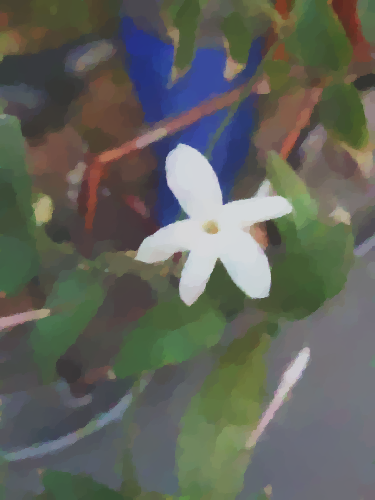}
		\includegraphics[width=.24\textwidth]{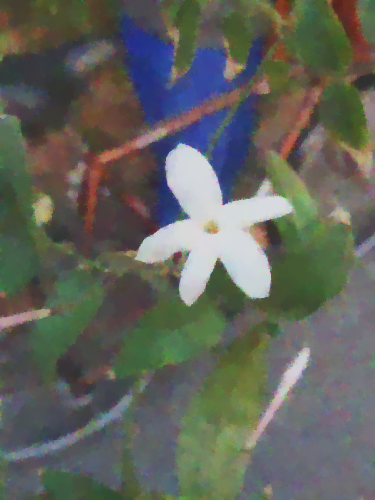}
		\caption{Same as Figure~\ref{fig:flower} with stronger noise.}\label{fig:flower_morenoise}
	\end{center}
\end{figure}
\begin{figure}[h]
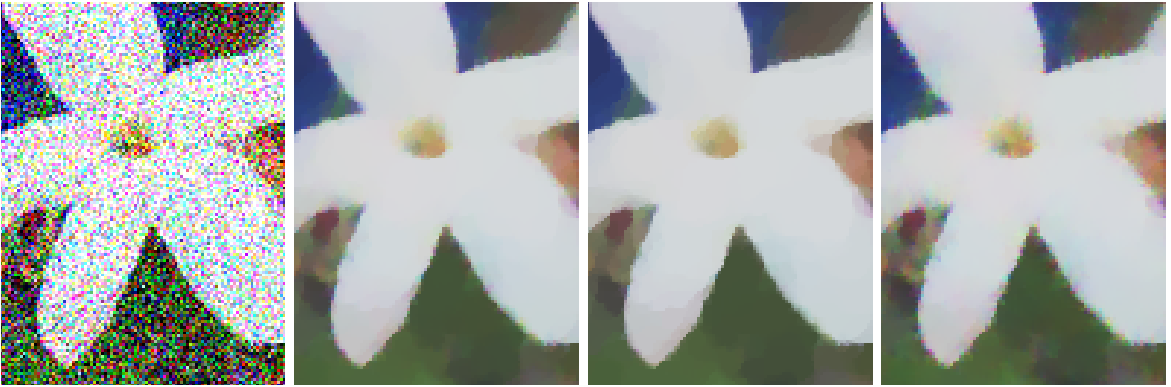

	\begin{center} 
		\includegraphics[trim=120 135 80 130,clip,width=.24\textwidth]{fleur_morenoise}
		\includegraphics[trim=120 135 80 130,clip,width=.24\textwidth]{fleur_fro_mn_4.png}
		\includegraphics[trim=120 135 80 130,clip,width=.24\textwidth]{fleur_nuc_mn_4.png}
		\includegraphics[trim=120 135 80 130,clip,width=.24\textwidth]{fleur_spe_mn_4.png}
		\caption{Detail of Figure~\ref{fig:flower_morenoise}}\label{fig:closeup_morenoise}
	\end{center}
\end{figure}

\section{Conclusion and comments}

We have introduced an approach for the study of the jump set of minimizers
of ``Rudin-Osher-Fatemi'' type problems which is more versatile than
the original approach in~\cite{Casellesetal2007},
and easier to handle than~\cite{Valkonen2015}
(even if equivalent in spirit, and much inspired by it).
We recover many cases (and more) from the previous
works~\cite{Valkonen2015,Valkonen2017}, including jump inclusion
in the ``TGV'' case (up to some smoothing). The full nonsmooth case
remains open. Also, our approach does not seem to allow to derive
further regularity, such as the continuity results of~\cite{Mercier,Casellesetal_regularity}. It is
also unclear
what exactly happens for non-differentiable norms, since at
jumps the gradients have rank one, and locally many nonsmooth
norms remain differentiable---yet experimental observations
seem to show a much worse control on the oscillations parallel
to the jumps in such cases, as in Figure~\ref{fig:closeup} (right).

Our results can be rephrased in terms of the measure $D^j u$. In particular, in the case of $\psi$ strongly convex with Lipschitz-gradient \eqref{sconv}, Theorems \ref{thethm} and \ref{thm:inf_conv} imply estimate $|D^j u| \leq \sqrt{\lambda/\Lambda} |D^j f|$. Similar bounds have been recently obtained for the whole singular part $D^s u$ in the 1D vectorial case ($m=1$, $n>1$) in \cite{GrochLasica}. To our knowledge, it remains an open question whether such estimates hold for the Cantor part $D^c u$ in $m>1$, even in the case of scalar $TV$, although in $n=1$ it is known that $D^s f = 0 \implies D^s u =0$ for general regularizers of form \eqref{R_rho} if $\Omega$ is convex \cite{LasicaRybka}, and that regularity away from the jump set is
transferred to the solution (for $\cR$ nice enough)~\cite{Casellesetal_regularity,Mercier}.

Finally, a natural question is whether the results shown in this work
also hold for the gradient flow of the total variation or similar
functionals. In~\cite{Casellesetal2007,CasellesJalalzaiNovaga}, this
is deduced from Crandall--Liggett's theorem in $L^\infty(\Omega)$,
which can be applied because minimizing the ROF problem~\eqref{eq:ROF}
is contractive in the sup norm. Yet, this is unknown (and possibly not
true) in the vectorial case and no easy conclusion may be drawn. In relation to this, we mention a recent paper \cite{Kazaniecki2023}, where an interesting continuity property of the map $w \mapsto D^j w$ is obtained. However, its applicability in our context remains a matter of further investigation.

\section*{Acknowledgements}

The authors are indebted to R\'emy Rodiac for inspiring conversations about differentiability along inner variations, and to Ir\`ene Waldspurger for helping sort out  
the properties of Schatten norms. We are also grateful to the anonymous reviewers for their comments and suggestions that helped improve the paper.
Figures \ref{fig:push}, \ref{fig:cylinders}, and \ref{fig:t72} were prepared using Wolfram Cloud and MS Paint.

\bibliographystyle{siamplain}
\bibliography{bib.bib}

\end{document}